%% file: hr2_e.tex
\pdfoutput=1
\documentclass[a4paper,leqno]{amsart}
\usepackage{amsmath,amsthm,amssymb,mathtools,mathrsfs} 
\usepackage{tikz-cd}

\usepackage[hidelinks,breaklinks]{hyperref}
\usepackage{stmaryrd} %
\usepackage[shortlabels]{enumitem}
\usepackage[utf8]{inputenc}
\usepackage{xcolor} %
\usepackage{a4wide}
\usepackage{microtype}

\input{preamble.tex}

\begin{document}

\title{BPS cohomology for rank 2 degree 0 Higgs bundles (and more)}
\author{Sebastian Schlegel Mejia}
\address{School of Mathematics, University of Edinburgh, Peter Guthrie Tait Road, Edinburgh EH9 3FD, United Kingdom}
\email{s.schlegel-mejia@sms.ed.ac.uk}

\begin{abstract}
We give a formula comparing the E-series of the moduli stacks of rank 2 degree 0 semistable Higgs bundles in genus $g \geq 2$ to intersection E-polynomials of its coarse moduli space. A parallel formula holds in various 2-Calabi--Yau settings, for example for sheaves on K3 surfaces, or preprojective algebras of $g$-loop quivers.
As a consequence we provide evidence for a conjecture of Davison on the BPS cohomology of Higgs bundles, which has implications for non-abelian Hodge theory for stacks.
We apply the formula to cohomological $\chi$-independence tests for BPS cohomology of Higgs bundles and K3 surfaces.
\end{abstract}
\maketitle

\section{Introduction}

The Borel--Moore homology of moduli stacks $\ModuliStack$
of objects in 2-Calabi--Yau categories 
as well as the compactly supported intersection cohomology of 
their coarse moduli spaces $\ModuliSpace$ 
have been at the intersection of many recent advances 
in enumerative invariants in algebraic geometry and 
geometric representation theory
\cite{davison2021Purity2CalabiYauCategories,
davison2020BPSLieAlgebras,
davison2021NonabelianHodgeTheory,
maulik2020EndoscopicDecompositionsHauselThaddeus,
maulik2021CohomologicalChiIndependence,
kapranov2022CohomologicalHallAlgebra,
kinjo2021CohomologicalChiIndependence,
kinjo2021DimensionalReductionCohomological,
kinjo2021GlobalCriticalChart,
sala2020CohomologicalHallAlgebra,
schiffmann2013CherednikAlgebrasWalgebras,
schiffmann2016IndecomposableVectorBundles}.

This paper is an attempt to 
explicitly describe the 
topology of the morphism $p\colon \ModuliStack \to \ModuliSpace$ to the coarse moduli space whenever it is locally modelled on the semi-simplification morphism
\begin{equation*}
\ModuliStack_{2}(\Pi_{S_g}) \to \ModuliSpace_2({\Pi_{S_g}})
\end{equation*}
of two-dimensional representations of the preprojective algebra
of the $g$-loop quiver $S_g$ 
for $g\geq 2$. 

More concretely, 
we prove a formula relating the compactly supported cohomology 
(the dual of the Borel--Moore homology)
of the moduli stack
$\ModuliStack$ with 
the compactly supported intersection cohomology of 
the coarse moduli space $\ModuliSpace$ for the following examples

\begin{enumerate}
\item certain moduli of semistable Higgs bundles on a smooth projective complex
curve $C$ of genus $g\geq 2$
\item certain moduli of $\pi_1(\Sigma_g \setminus \{p\})$-representations 
with prescribed monodromy around the puncture $p$ for a 
closed Riemann surface $\Sigma_g$ of genus $g \geq 2$
\item certain moduli of sheaves on K3 or abelian surfaces
\item moduli of two-dimensional representations 
of the preprojective algebra $\Pi_{S_g}$ of the $g$-loop
quiver $S_g$ for $g\geq 2$.
\end{enumerate}
See Theorem~\ref{thm-general_e-poly} and
Corollary~\ref{cor-general_coh} for a precise formula.

\subsection{Cohomology of moduli spaces of 
rank 2 Higgs bundles}

Let $C$ be a smooth connected projective complex curve 
of genus $g \geq 2$.
Let 
$\ModuliStack^\Dol_{r,d}$
be the moduli stack of rank $r$ degree $d$ semistable Higgs bundles on $C$ 
and let 
$p_{r,d}\colon \ModuliStack^{\Dol}_{r,d} \to \ModuliSpace^{\Dol}_{r,d}$
be its coarse moduli space.

When $\gcd(r,d) = 1$, the moduli space $\ModuliSpace^{\Dol}_{r,d}$ and
moduli stack $\ModuliStack_{r,d}^{\Dol}$ are smooth 
and the morphism $p_{r,d}$  is the trivial $\GG_m$-gerbe over $\ModuliSpace^{\Dol}_{r,d}$ 
\begin{equation*}
p_{r,d}\colon \ModuliStack_{r,d}^{\Dol} = \ModuliSpace_{r,d}^{\Dol}\times B\GG_m 
\longto \ModuliSpace_{r,d}^{\Dol}.
\end{equation*}
By the K\"unneth formula, the compactly supported cohomology of the moduli stack is determined from the compactly supported cohomology of the moduli space. The latter has been studied extensively (see for example \cite{garcia-prada2014MotivesModuliChains}).

When $\gcd(r,d) \neq 1$, the moduli space $\ModuliSpace_{r,d}^{\Dol}$ and moduli 
stack $\ModuliStack_{r,d}^{\Dol}$ are singular
and the morphism $p_{r,d}\colon \ModuliStack_{r,d}^{\Dol} \to 
\ModuliSpace_{r,d}^{\Dol}$ is more complicated than in the coprime case.
In the literature the cohomology of the moduli stack and moduli space have been
studied separately.

The compactly supported intersection cohomology of the moduli spaces $\IHc(\ModuliSpace_{r,d}^{\Dol})$
has been determined for rank 2 and degree 0
in genus $2$ by Felisetti \cite{felisetti2021IntersectionCohomologyModuli} 
and  for arbitrary genus $ g \geq 2$ by Mauri 
\cite{mauri2021IntersectionCohomologyRank}.
Building on work of Schiffmann \cite{schiffmann2016IndecomposableVectorBundles} and Mozgovoy--Schiffmann \cite{mozgovoy2020CountingHiggsBundles} on point counts of the stacks $\ModuliStack_{r,d}^{\Dol}$ over finite fields, Fedorov--Soibelman--Soibelman  in \cite{fedorov2018MotivicClassesModuli} 
find a formula for the class $[\ModuliStack_{r,d}^{\Dol}]$ in the 
Grothendieck ring of stacks $K_0(\Stacks)$.

Set $d=0$, guaranteeing $\gcd(r,d) \neq 1$ for $r > 1$. Consider the 
compactly supported cohomology (with rational coefficients) for all ranks
\begin{equation*}
\CoHa = 
\bigoplus_{r \geq 0} \Hc^\bullet(\ModuliStack^\Dol_{r,0}) 
\otimes \LL^{(1-g)r^2}
\end{equation*}
as an object in
the symmetric monoidal abelian category of (graded) mixed Hodge structures. 
Here $\LL = \Hc^{\bullet}(\AA^1)$ is the mixed Hodge structure 
given by the compactly supported cohomology of the affine line.

Motivated by cohomological Donaldson--Thomas theory
and non-abelian Hodge theory
for stacks
Davison conjectured the following for $\CoHa$.

\begin{conjecture}[{\cite[Conjecture~5.6]{davison2021NonabelianHodgeTheory}}]
\label{conj}
There is an isomorphism of mixed Hodge structures
\begin{equation*} 
\CoHa \cong 
\Sym
\bigg(
\FreeLie
\bigg(\bigoplus_{r\geq 1} \IHc^{\bullet}(\ModuliSpace^{\Dol}_{r,0}) 
\otimes \LL^{(1-g)r^2-1}
\bigg) 
\otimes \Hc^{\bullet}(B\GG_m) \otimes \LL
\bigg)
\end{equation*}
where $\Sym$ is taken in the graded sense and 
$\FreeLie(V^\bullet)$ denotes the free graded 
Lie algebra generated by 
the graded (super) vector space $V^\bullet$.
\end{conjecture}

\begin{remark}
Conjecture~4.6 in \cite{davison2021NonabelianHodgeTheory} is the corresponding conjecture on the other side of non-abelian Hodge theory. Conjecture 7.7 in \cite{davison2020BPSLieAlgebras} is a related conjecture in the setting of preprojective algebras of quivers.
\end{remark}

The original motivation for this work was to 
provide evidence towards this conjecture.
Conjecture~\ref{conj} predicts the following formula for the
compactly supported cohomology of the moduli stack of rank 2
degree 0 Higgs bundles.

\begin{theorem}
\label{thm-main_mixedhodgestructure}
There is an isomorphism of graded mixed Hodge structures
\begin{equation*}
\begin{split}
\Hc^\bullet(\ModuliStack^\Dol_{2,0}) \otimes \LL^{4-4g} 
&\cong
\IHc^\bullet(\ModuliSpace^\Dol_{2,0}) \otimes \LL^{3-4g} 
\otimes \Hc^\bullet(B\GG_m) \otimes \LL\\
&\qquad
\oplus \Lambda^2(\Hc^\bullet(\ModuliSpace_{1,0}^{\Dol})
\otimes \LL^{-g}) 
\otimes \Hc^{\bullet}(B\GG_m) \otimes \LL
\\
&\qquad
\oplus \Sym^2(\Hc^{\bullet}(\ModuliSpace_{1,0}^{\Dol})
\otimes \LL^{-g}\otimes \Hc^{\bullet}(B\GG_m)\otimes \LL).
\end{split}
\end{equation*}
where the alternating square $\Lambda^2$ and the symmetric square $\Sym^2$ 
are taken in the symmetric monoidal abelian category of graded mixed Hodge structures.
\end{theorem}

Let $q=E(\LL)$. We deduce Theorem~\ref{thm-main_mixedhodgestructure} from the following equality of E-series
(see Corollary~\ref{cor-general_coh}).
\begin{theorem}
\label{thm-main_e-polynomial}
\begin{equation*}
\begin{split}
\frac{E(\ModuliStack^\Dol_{2,0})}{\ELL^{4g-4}} &=
\frac{1}{\ELL^{4g-3}}E(\IHc^{\bullet}(\ModuliSpace_{2,0}^\Dol))
E(\Hc^\bullet(B\GG_m) \otimes \LL) 
\\
&\qquad+
\frac{1}{\ELL^{2g}}E(\Lambda^2(\Hc^\bullet(\ModuliSpace_{1,0}^{\Dol})))
E(\Hc^\bullet(B\GG_m) \otimes \LL)
\\
&\qquad +
\frac{1}{\ELL^{2g}}E(\Sym^2(\Hc^\bullet(\ModuliSpace_{1,0}^{\Dol})\otimes \Hc^\bullet(B\GG_m) \otimes \LL ))
\end{split}
\end{equation*}
\end{theorem}
\begin{remark}
	The RHS can be made more explicit using $E(\Hc^\bullet(B\GG_m) \otimes \LL) = q/q-1$ and Lemma~\ref{lm-epoly_quot_stacks}.
\end{remark}
\subsection{Other settings}

Although first intended for moduli of semistable rank 2 degree 0 Higgs bundles, the calculation we give in Section~\ref{sec-stack_e-poly} 
works in many other settings as well.
 
To get an analogue of the formula 
in Theorem~\ref{thm-main_e-polynomial} for other settings
we need as input an integer $g\geq 2$,
 a stack $\ModuliStack_2^{\sst}$, 
and spaces $\ModuliSpace_2^{\sst}, \ModuliSpace_1^{\st}$
which play the role of 
$\ModuliStack^{\Dol}_{2,0}$ and 
$\ModuliSpace_{2,0}^{\Dol},\ModuliSpace_{1,0}^{\Dol}$, respectively.

\begin{setting}
[More Higgs bundles] 
\label{set-Higgs}
Let $(r,d) \in \ZZ_{\geq 1} \times \ZZ$ such that
$\gcd(r,d) = 1$. We consider rank $2r$ and degree $2d$
semistable Higgs bundles 
on a smooth connected projective curve $C$ of genus $g_0 \geq 2$ . In this setting we take 
\begin{equation*}
g = r^2(g_0-1)-1,\; 
\ModuliStack_2^{\sst} =\ModuliStack_{2r,2d}^{\Dol}, \;
\ModuliSpace_2^{\sst} = \ModuliSpace_{2r,2d}^{\Dol},\; 
\ModuliSpace_1^{\st} = \ModuliSpace_{r,d}^{\Dol}.
\end{equation*}
\end{setting}

\begin{setting}
[Character stacks]
\label{set-pi_1-mod}
Let $\Sigma_{g_0}$ be a closed connected Riemann surface of genus $g_0\geq 2$ 
and fix a point $p\in \Sigma_{g_0}$.
Let $\ModuliStack_{g_0,r,d}^{\Betti}$ be the moduli stack or
$r$-dimensional representations of 
$\pi_1(\Sigma_{g_0} \setminus \{p\})$ with
monodromy around the puncture given by $e^{2d \pi i / r}$ and
let $\ModuliSpace_{g_0,r,d}$ denote its coarse moduli space.
Let $(r,d) \in \ZZ_{\geq 1} \times \ZZ$ such that $\gcd(r,d)=1$.
In this setting we take
\begin{equation*}
g = r^2(g_0-1)-1,\;
\ModuliStack_2^{\sst} = \ModuliStack^{\Betti}_{g_0,2r,2d}, \;
\ModuliSpace_2^{\sst} = \ModuliSpace^{\Betti}_{g_0,2r,2d}, \;
\ModuliSpace_1^{\st} = \ModuliSpace^{\Betti}_{g_0,r,d}.
\end{equation*}
\end{setting}

\begin{setting}
[K3 and abelian surfaces] 
\label{set-cpt2CY}
Let $S$ be a K3 or abelian surface
and $H$ an ample class on $S$. 
For every Mukai vector $v \in H^{\bullet}_{\mathrm{alg}}(S,\ZZ)$
let $\ModuliStack^{\Hsst}_{S,v}$ and $\ModuliSpace^{\Hsst}_{S,v}$
be the moduli stack and moduli space, respectively,
of $H$-semistable sheaves with Mukai vector $v$ on $S$.
Consider a primitive Mukai vector 
$w \in \H^{\bullet}_{\mathrm{alg}}(S,\ZZ)$
with $w^2 \geq 0$ such that $H$ is generic with respect to $w$.
In this setting we take 
\begin{equation*}
g = \frac{w^2+2}{2},\;
\ModuliStack_{2}^{\sst} = \ModuliStack^{\Hsst}_{S,2w},\;
\ModuliSpace_2^{\sst} = \ModuliSpace^{\Hsst}_{S,2w}, \;
\ModuliSpace_1^{\st} = \ModuliSpace^{\Hst}_{S,w}.
\end{equation*}
\end{setting}

\begin{setting}
[Preprojective algebra of the $g$-loop quiver]
\label{set-preproj}
Let $S_g$ be the $g$-loop quiver.
Let $\ModuliStack_{d}(\Pi_{S_g})$ be the moduli stack 
and $\ModuliSpace_{d}(\Pi_{S_g})$ the coarse moduli space of 
$d$-dimensional representations of the preprojective algebra of $S_g$. 
In this setting we take
\begin{equation*}
g = g, \;
\ModuliStack_{2}^{\sst} = \ModuliStack_2(\Pi_{S_g}),\;
\ModuliSpace_{2}^{\sst} = \ModuliSpace_{2}(\Pi_{S_g}),\;
\ModuliSpace_{1}^{\st} = \ModuliSpace_{1}(\Pi_{S_g}).
\end{equation*}
\end{setting}

\begin{theorem}
\label{thm-general_e-poly}
Let $(g,\ModuliStack_{2}^{\sst},\ModuliSpace_{2}^{\sst},\ModuliSpace_{1}^{\st})$ be as in one of Settings~1-4. Then

\begin{equation*}
\begin{split}
\frac{E(\ModuliStack_{2}^\sst)}{\ELL^{4g-4}} &=
\frac{1}{\ELL^{4g-3}}E(\IHc^{\bullet}(\ModuliSpace_{2}^\sst))
E(\Hc^\bullet(B\GG_m)\otimes \LL) 
\\
&\qquad+
\frac{1}{\ELL^{2g}}E(\Lambda^2(\Hc^\bullet(\ModuliSpace^\sst_{1})))
E(\Hc^\bullet(B\GG_m) \otimes \LL)
\\
&\qquad +
\frac{1}{\ELL^{2g}}E(\Sym^2(\Hc^\bullet(\ModuliSpace_{1}^\sst)\otimes \Hc^\bullet(B\GG_m)\otimes \LL)).
\end{split}
\end{equation*}
\end{theorem}

\begin{remark}
	A key ingredient for the proof is Theorem~1.3 in \cite{mauri2021IntersectionCohomologyRank} which provides a ``uniform'' formula for the intersection E-polynomial of the moduli spaces $\ModuliSpace_2^{\sst}$ for each of the settings. See Proposition~\ref{thm-Mauri-computation}. 
\end{remark}

\subsection{BPS cohomology}
\label{sec-BPS}
As mentioned, the motivation for Conjecture~\ref{conj} 
comes from 
cohomological Donaldson--Thomas theory for 2CY-categories. 
In particular Conjecture~\ref{conj} can be interpreted as a statement about the BPS-cohomology for Higgs bundles.
In this section we give a short introduction to BPS-cohomology of 2CY-categories.

The foundations of cohomological Donaldson--Thomas theory for 2-Calabi--Yau categories have been laid out in \cite{davison2021Purity2CalabiYauCategories,
davison2017IntegralityConjectureCohomology,
davison2020BPSLieAlgebras,
davison2021NonabelianHodgeTheory,
davison2016CohomologicalHallAlgebras,
kinjo2021CohomologicalChiIndependence} 
and we refer the reader to these papers for more details.
\\

Let $\sC$ be a 2-Calabi--Yau (2CY) abelian category,
i.e., 
there are natural non-degenerate graded-symmetric pairings
\begin{equation}
\label{eq-2CY}
\Ext_{\sC}^{\bullet}(E,F) \times \Ext_{\sC}^{2-\bullet}(F,E) \longto \CC
\end{equation}
for all $E,F \in \sC$.

By the 2CY-property, the Euler pairing $\chi(E,F) = \chi(\Ext_{\sC}^{\bullet}(E,F))$,
for $E,F \in \sC$,
defines a symmetric pairing on the Grothendieck group $K_0(\sC)$.
Assume the Euler pairing is the pullback along a group homomorphism $\gamma\colon K_0(\sC) \to \ZZ^n$
of a symmetric bilinear form on $\ZZ^n$,
which by abuse of notation we denote by $\chi$.

Let $v \in \ZZ^n$ be a primitive element.
For every $r \geq 0$ let $\ModuliStack_{rv}$ be
the moduli stack of objects in $E \in \sC$ of class $rv$
, that is, $\gamma(E) = rv$. 

We require a notion of \emph{semistable} objects in the category $\sC$, which is additive in short exact sequences: if $E'\into E \onto E''$ is a short exact sequence with $E',E''$ semistable, then $E$ is also semistable.
Imposing semistability must define open substacks of finite type
$\ModuliStack_{rv}^{\sst} \subset \ModuliStack_{rv}$ that  admit a good moduli space
$\ModuliStack^{\sst}_{rv} \to \ModuliSpace_{rv}^{\sst}$ and have virtual dimension $\vdim(\ModuliStack_{rv}^\sst) = r^2\chi(v,v)$.

The total cohomology
\begin{equation*}
\CA_{\sC,v}^{\sst} = \bigoplus_{r\geq 0} \Hc^{\bullet}(\ModuliStack_{rv}^{\sst}) \otimes \LL^{\vdim(\ModuliStack_{rv}^{\sst})/2}
\end{equation*}
is an important object of study in the subject of cohomological Donaldson--Thomas theory.

\begin{definition}\label{def-cohintconj}
We say that $\CA_{\sC,v}^{ss}$ \emph{satisfies 
the cohomological integrality conjecture} if there
exists mixed Hodge structures $\BPS_{rv}$
of finite total dimension
and an isomorphism of 
$\ZZ_{\geq 0}$-graded mixed Hodge structures 
\begin{equation*}
\CA_{\sC,v}^{\sst} \cong \Sym\bigg(\bigoplus_{r\geq 1} \BPS_{rv} 
\otimes \Hc^{\bullet}(B\GG_m)\otimes \LL \bigg).
\end{equation*}
The graded mixed Hodge structure $\bigoplus_{r \geq 1}\BPS_{rv}$ is
the \emph{BPS cohomology}.
\end{definition}

\begin{remark}
The cohomological integrality conjecture is known in 
Setting~\ref{set-Higgs} by \cite{kinjo2021CohomologicalChiIndependence},
Setting~\ref{set-pi_1-mod} by \cite{davison2016CohomologicalHallAlgebras}, and
Setting~\ref{set-preproj} by \cite{davison2017IntegralityConjectureCohomology}.
\end{remark}

Theorem~\ref{thm-general_e-poly} can be rephrased as a formula for the E-polynomial of BPS cohomology.
\begin{corollary}
\label{cor-e_BPS_gen}
Suppose $\CA_{\sC,v}^{\sst}$ satisfies the cohomological integrality conjecture. Then
\begin{equation*}
E(\BPS_{2v}) = \frac{E(IH_c^{\bullet}(\ModuliSpace_{2}^{\sst}))}{\ELL^{4g-3}}+
E(\Lambda^2(\Hc(\ModuliSpace_{1}^{\sst})\otimes \LL^{-g}))  .
\end{equation*}
\end{corollary}

Moreover, assuming purity of $\CA^{\sst}_{\sC,v}$, we can upgrade
the statement for E-polynomials to a statement for cohomology.
\begin{corollary}
\label{cor-general_coh}
Suppose $\CA^{\sst}_{\sC,v}$ satisfies the cohomological integrality conjecture.
If $\Hc^\bullet(\ModuliStack_2^{\sst})$ is pure, then there is a 
canonical isomorphism of 
mixed Hodge structures
\begin{equation*}
\begin{split}
\Hc^\bullet(\ModuliStack_{2}^{\sst}) \otimes \LL^{4-4g} 
&\cong
\IHc^\bullet(\ModuliSpace_{2}^{\sst}) \otimes \LL^{3-4g} 
\otimes \Hc^\bullet(B\GG_m)\otimes \LL \\
&\qquad
\oplus \Lambda^2(\Hc^\bullet(\ModuliSpace_{1}^{\st})) \otimes \Hc^{\bullet}(B\GG_m) \otimes \LL
\\
&\qquad
\oplus \Sym^2(\Hc^{\bullet}(\ModuliSpace_{1}^{\st})\otimes \Hc^{\bullet}(B\GG_m) \otimes \LL),
\end{split}
\end{equation*}
equivalently, there is a canonical isomorphism
\begin{equation*}
\BPS_{2v} \cong IH_c^{\bullet}(\ModuliSpace_{2}^{\sst}) \otimes \LL^{3-4g}
\oplus \Lambda^2(\Hc(\ModuliSpace_{1}^{\sst}\otimes \LL^{-g})).
\end{equation*}
\end{corollary}
\begin{proof}
By the cohomological integrality theorem $\H_c^\bullet(\ModuliStack_2^{\sst}) \otimes \LL^{4-4g}$ must be the 
second graded piece of 
$\Sym(\bigoplus_{r\geq 1}\BPS_{rv} \otimes \Hc^\bullet(B\GG_m) \otimes \LL)$.
Thus by 
\cite[Theorem~6.6]{davison2021Purity2CalabiYauCategories}
there is a canonical inclusion
of the RHS into $\H_c^\bullet(\ModuliStack_2^{\sst}) \otimes \LL^{4-4g}$. Purity and equality of E-polynomials (Theorem~\ref{thm-general_e-poly}) implies that the inclusion is in
fact an isomorphism of mixed Hodge structures.
\end{proof}
\begin{example}
In Settings \ref{set-Higgs}, \ref{set-cpt2CY},
and \ref{set-preproj} the compactly supported cohomology
$\Hc^\bullet(\ModuliStack_2^{\sst})$ is pure 
(\cite{davison2017IntegralityConjectureCohomology,
davison2021Purity2CalabiYauCategories}).
However, in Setting \ref{set-pi_1-mod} 
the compactly supported cohomology 
$\Hc^\bullet(\ModuliStack_{2}^{\sst})$ is not pure.
\end{example}

\subsection{Notation}
We work over the complex numbers. 
For an algebraic group $G$ acting
on a scheme $X$ denote by $X/G$ the quotient stack.

We always take (compactly supported) cohomology and intersection cohomology with rational coefficients.
We overload the notation $\LL$ which denotes the mixed Hodge structure
$\Hc^\bullet(\AA^1)$ or the class $[\AA^1]$ in the Grothendieck ring of varieties. 
The E-polynomial $\ELL \coloneqq E(\LL)=uv$ is unambiguous.

\subsection*{Acknowledgements}
I would like to thank my advisor Ben Davison for suggesting the project and for many helpful discussions. I would like to thank Lucien Hennecart, Naoki Koseki, and Mirko Mauri for useful conversations. I am grateful to Francesca Carocci and Dimitri Wyss for giving the opportunity to give a talk at EPFL on the results of the paper and for ensuing discussions.
Many thanks are also due to the anonymous referee for a thorough and constructive review.

This research was supported by the ERC starter grant ``Categorified Donaldson–Thomas theory'' No. 759967 of the European research council.

\section{E-series}
\label{prelim}

\label{sec-E-series}

\subsection{E-series of mixed Hodge structures}
A \emph{(cohomologically graded) mixed Hodge structure} 
is a triple $H=(H^\bullet,W,F)$ consisting of 
\begin{itemize}
\item 
a $\ZZ$-graded vector space $H$ over $\QQ$
\item 
an increasing \emph{weight filtration} $W$ on the graded vector space $H$
\item 
a decreasing \emph{Hodge filtration} $F$
on the complexified graded vector space $H_\CC$
\end{itemize}
such that the filtration $F$ on the complexification of each
associated graded piece $W_{k}H/W_{k-1}H$ endows said piece with a rational weight $k$ pure Hodge structure. 

A cohomologically graded mixed Hodge structure $H$ is said to be a \emph{(cohomologically graded) pure Hodge structure} if $H^n = (H^n,W\cap H^n, F \cap H^n)$ is a (non-graded) pure Hodge structure of weight $n$ for all $n \in \ZZ$.

Henceforth all Hodge structures are assumed
to be cohomologically graded.
\\

For every cohomologically graded mixed Hodge structure $H=(H^\bullet,W,F)$ such that 
\begin{equation}\label{eq-E-series_hyp}
\dim(\Gr_{F}^p \Gr^W_{p+q}H^n) < \infty\text{, }
H^n = 0 \text{ for } n \gg 0,\text{ and }
\Gr_{p}^W H^n = 0 \text{ for } p > n
\end{equation}
we define the \emph{E-series} of $H$ to be 
\begin{equation*}
E(H) = \sum_{p,q \in \ZZ} \sum_{n \in \ZZ} (-1)^n
\dim (\Gr_{F}^{p} \Gr^W_{p+q} H^{n}) u^p v^q \in \ZZ((u^{-1},v^{-1})).
\end{equation*}
If $H$ is finite dimensional, then $E(H)$ is a Laurent-polynomial 
in $u,v$ and so we call it the \emph{E-polynomial} of $H$.

\begin{remark}
	The boundedness condition \eqref{eq-E-series_hyp} is satisfied for the compactly supported cohomology for all finite type Artin stacks with affine stabilisers \cite[Lemma~4.6]{davison2020BPSLieAlgebras}. Moreover it guarantees that the E-series is well-defined.
\end{remark}

Let $\MHSboundedabove$ be the category of 
mixed Hodge structures 
$H=(H^\bullet,W,F)$ satisfying \eqref{eq-E-series_hyp}. 
It is a symmetric monoidal abelian category. 
The E-series defines a ring homomorphism
\begin{equation*}
E\colon K_0(\MHSboundedabove) \longto \ZZ((u^{-1},v^{-1})).
\end{equation*}

\subsection{E-series of varieties}
For every finite type separated scheme $X$ over $\CC$ 
its compactly supported cohomology $\Hc^\bullet (X)$ is endowed with a mixed Hodge structure by
\cite{deligne1971TheorieHodgeII,deligne1974TheorieHodgeIII}.
The \emph{E-polynomial} of a finite type separated scheme $X$ 
is the E-polynomial of its compactly supported cohomology 
$E(X) = E(\Hc^\bullet (X)).$
The E-polynomial is a motivic invariant: for every closed subscheme $Z \subset X$ we have $E(X) = E(Z) + E(X \setminus Z)$ and $E(X \times Y) = E(X)E(Y)$.
\begin{example}
$\ELL \coloneqq E(\LL) = uv$, $E(\GG_m)= \ELL-1$, $E(\PP^1)=\ELL+1$, $E(\GL_2)=\ELL(\ELL+1)(\ELL-1)^2$, and $E(\GL_n) = \prod_{i=0}^{n-1} (\ELL^n-\ELL^{i})$. 
\end{example}

\subsection{E-series of stacks}
There are two equivalent approaches to defining 
the E-series of a quotient stack. 
The first uses the Grothendieck ring of stacks.
The second uses an algebro-geometric approximation of the Borel construction
to define a mixed Hodge 
structure on the compactly supported cohomology of a quotient stack, 
see \cite[Section~2.3]{davison2021Purity2CalabiYauCategories} for details.
We take the approach 
via the Grothendieck ring of stacks.

Let $K_0(\Varieties)$ and $K_0(\Stacks)$, denote the
Grothendieck ring of varieties and
finite type Artin stacks with affine stabilizers, respectively.
See
\cite{ekedahl2009GrothendieckGroupAlgebraic,
fedorov2018MotivicClassesModuli,
garcia-prada2014MotivesModuliChains}
for introductions to these rings.
Let $\LL = [\AA^{1}]$ be the class of the affine line.  
\begin{proposition}
[{\cite[Theorem~1.2]{ekedahl2009GrothendieckGroupAlgebraic}}]
\label{prop-K0_var_stacks_comp_ft}
The classes $\LL=[\GG_a]$, $[\GL_n]$ in $K_0(\Stacks)$ are invertible with inverses $[B\GG_a]$, resp., $[B \GL_n]$,
and the natural inclusion
\begin{equation*}
\FI \colon{K_0}(\Varieties)[\LL^{-1},[\GL_n]^{-1}] \longinto K_0(\Stacks)
\end{equation*}
is an isomorphism of rings. 
\end{proposition}
We use this isomorphism to define the E-series of 
a finite type stack with affine stabilizers.
\begin{definition}
\label{def-E-series_stack}
By the K\"unneth formula and 
the long exact sequence in compactly supported cohomology,
the E-polynomial of a variety defines a ring homomorphism
\begin{equation*}
E\colon K_0(\Varieties) \longto \ZZ[u,v],
\end{equation*}
which induces the
\emph{E-series} homomorphism 
\begin{equation*}
E= E\circ \FI^{-1} \colon K_0(\Stacks) \longto
\ZZ[u,v][(uv)^{-1},E(\GL_n)^{-1}] \subset \ZZ((u^{-1},v^{-1})),
\end{equation*} 
where we overload $E$ in our notation. 
\end{definition}

\begin{example}
\label{ex-E_quotbyfin}
In general $E(X/G) \neq E(X)/E(G)$.
Let $G$ be a finite group acting on a smooth variety $X$.
Then $E(X/G) =E(\Hc^\bullet (X)^G)$. %
\end{example}

\begin{example}
An algebraic group $G$ is \emph{special} if every $G$-torsor is 
Zariski-locally trivial.
Let $G$ be a special algebraic group acting on a variety $X$.
The E-series of the quotient stack $X/G$ is 
$
E(X/G) = E(X)/E(G).
$
\end{example}
\begin{example}
	Suppose $G$ is a special algebraic group and $\FX \to X$ is a $G$-gerbe.
	Then $E(\FX) = E(X)/E(G)$
\end{example}

\subsection[E-series of a symmetric square via lambda-rings]{E-series of a symmetric square via \texorpdfstring{$\lambda$}{lambda}-rings}
For our computations in the sequel it is convenient to use the language of 
$\lambda$-rings. 
We refer to \cite{knutson1973RingsRepresentationTheory} 
for the basics of $\lambda$-rings.

The ring $R=\ZZ((u^{-1},v^{-1}))$ carries 
two natural $\lambda$-ring structures
coming from the isomorphism $R\cong K_0(\mathbf{Vect}_{\ZZ^2}^{-})$ 
where $\mathbf{Vect}_{\ZZ^2}^{-}$ is the symmetric monoidal abelian category 
of $\ZZ^2$-graded, bounded above vector spaces
with finite dimensional graded pieces. 
The first is the $\lambda$-ring structure $\lambda(t) = \sum_{n \geq 0} \lambda^n t^n$
induced by taking alternating powers $\Lambda^n V$ 
for  $V \in \Ob(\mathbf{Vect}_{\ZZ^2}^{-})$.
The second is the \emph{symmetric power} $\lambda$-ring structure 
$\sigma(t) = \sum_{i\geq 0}\sigma^n t^n$ 
induced by taking symmetric powers $\Sym^{n}V$ 
for $V \in \Ob(\mathbf{Vect}_{\ZZ^2}^{-})$.
The $\lambda$-ring structures $\lambda$ and $\sigma$ in $R$ are opposite.

In the sequel we only require $\sigma^2$ and $\lambda^2$,
which are explicitly given for all $f\in R = \ZZ((u^{-1},v^{-1}))$ by
\begin{equation*}
\sigma^2(f) = \frac{1}{2}(f(u,v)^2 + f(u^2,v^2)) \quad \text{ and } \quad
\lambda^2(f) = \frac{1}{2}(f(u,v)^2 - f(u^2,v^2)).
\end{equation*}

The ring $K_0(\Stacks)$ admits 
a pre-$\lambda$-ring structure 
$\MSym[\FX](t) = \sum_{i\geq 0} \MSym^n[\FX]t^i$
given by the (stacky) symmetric powers
$
\MSym^i[\FX] = [\MSym^i(\FX)].
$
The opposite pre-$\lambda$-ring structure
$\MAlt[\FX](t) = \sum_{i \geq 0} \MAlt^n[\FX]t^n$
is given by 
$
\MAlt[\FX](t) = (\MSym(\FX)(-t))^{-1}
$.
We think of $\MAlt^n[\FX]$ as the class of 
the $n$th alternating power of $[\FX]$.
Indeed, for all varieties $X$ we have the equality 
\begin{equation*}
E(\MAlt^n[X]) = E(\Lambda^n\Hc(X)).
\end{equation*}
The E-series ring homomorphism of
Definition~\ref{def-E-series_stack} is a homomorphism of
pre-$\lambda$-rings,
because $\FI$ is an isomorphism of pre-$\lambda$-rings
(see Example~3.5 and Proposition 3.6 in 
\cite{davison2011motivic}). 

With the formalism of $\lambda$-rings we easily find a concise
expression for the E-series of the symmetric square of a quotient by $\GG_m$.
\begin{lemma}
\label{lm-epoly_quot_stacks}
Let $\GG_m$ act on a separated scheme of finite type $X$. Then 
\begin{equation*}
E(\Sym^2(X/\GG_m)) = 
\frac{\ELL E(\MSym^2(X))+E(\MAlt^2[X])}{(\EGm)^2(\EPP)}.
\end{equation*}
\end{lemma}
\begin{proof}
\begin{gather*}
E(\Sym^2(X/\GG_m)) = \sigma^2(E(X/\GG_m)) 
= \frac{1}{2}\left(E(X/\GG_m)(u,v)^2 + E(X/\GG_m)(u^2,v^2)\right)
\\
= \frac{1}{2}\left(\frac{E(X)^2}{(\EGm)^2} 
+ \frac{E(X)(u^2,v^2)}{(\ELL^2-1)}\right) 
= \frac{\ELL\sigma^2(E(X))+\lambda^2(E(X))}{(\EGm)^2(\ELL+1)} 
\end{gather*}
\end{proof}
\begin{remark}
Upon reading the calculation in Section~\ref{sec-stack_e-poly}, the reader might notice that 
most of the steps are valid more generally in the Grothendieck ring of stacks. 
The main reason we do not work in $K_0(\Stacks)$ is to have access to 
the identity of Lemma~\ref{lm-epoly_quot_stacks}.
\end{remark}

\section{The computation}
\label{sec-stack_e-poly}
We continue to use the terminology and notation from Section~\ref{sec-BPS}. 
In this section we compute the E-series of certain moduli stacks $\ModuliStack^{\sst}_{2v}$ of semistable objects in certain 2CY-categories $\sC$. 
First we spell out the assumptions we make on the category $\sC$, the class $v$, semistability, and the moduli stacks and spaces, all of which are satisfied by each of our Settings~1-4.

We consider the moduli stack 
$\ModuliStack^{\st}_{1} = \ModuliStack_{v}^{\st}$ 
of semistable objects $L \in \sC$ of a primitive class $v$.
In all of our cases $\ModuliStack_{1}^{\st}$ is	 smooth 
and admits a smooth good moduli space 
$p_1\colon \ModuliStack_1^{\st} \to \ModuliSpace_1^{\st}$.
Since $v$ is primitive,
an object $L \in \sC$ of class $v$ is necessarily simple.
\begin{definition}
	For convenience 
	we call simple and semistable objects \emph{stable}. (This is why we write $\ModuliStack_{1}^{\st}$ and $\ModuliSpace_{1}^{\st}$ instead of $\ModuliStack_{1}^{\sst}$ and $\ModuliSpace_1^{\sst}$.)
	For a primitive class $v$, direct sums of stable objects all of which are of a class which is a multiple of $v$ are called \emph{polystable} of slope $v$.
\end{definition}

We also consider the moduli stack $\ModuliStack_{2}^{\sst} = \ModuliStack_{2v}^{\sst}$ of semistable objects of class $2v$.
The moduli space $\ModuliSpace_{2}^{\sst}$ is singular and parametrizes objects of class $2v$ which are both semistable and polystable. 
Semistable objects $E$ of class $2v$ are either stable or there is a stable subobject $K\subsetneq E$ of class $v$. 
In the second case, the isomorphism class of the direct sum $K \oplus E/K$ does not depend on the choice of the subobject $K$ of class $v$.
The morphism $p_2\colon \ModuliStack_{2}^{\sst} \to \ModuliSpace_{2}^{\sst}$ sends an object $E \in \ModuliStack_{2}^{\sst}$ to the factors of the filtration by stable objects of class $v$ or $2v$
\begin{equation*}
p_2(E) = \begin{cases}
E & \text{if } E \text{ is stable} \\
K \oplus E/K & \text{if } K \subsetneq E \text{ is of class }v
\end{cases}.
\end{equation*}

We assume that 
for every non-zero object $E\in \sC$ the first self-Ext-group does not vanish $\Ext^1_{\sC}(E,E) \neq 0$.
The 2CY-pairing \eqref{eq-2CY} restricts to
a non-degenerate alternating pairing on $\Ext^1_{\sC}(E,E)$ and 
so $\Ext^1_{\sC}(E,E)$ is even-dimensional. 
For all objects $L \in \sC$ of class $v$, let $g>1$ be the integer such that $\dim \Ext^{1}_{\sC}(L,L) = 2g$.
This implies that for all non-isomorphic objects $Q,K \in \sC$ of class $v$ we have 
$\dim \Ext_{\sC}^{1} (Q,K) = 2g-2$. 

\begin{remark}
[Ext-quivers]
\label{rm-ext-quivers}
All semistable objects of class $2v$ have one of the following three 
Ext-quivers.
    
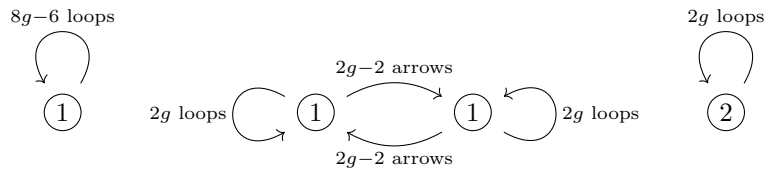
\begin{equation*}
	\begin{tikzcd}[column sep = large]
		\circled{1}  \ar[loop, distance= 3em,in = 120, out=60,"8g-6 \text{ loops}"'] & &
		\circled{1} \ar[loop,in = 210, out=150,distance=3em,"2g \text{ loops}"'] \ar[r,bend left,"2g-2 \text{ arrows}"]& \circled{1}
		\ar[l,bend left, "2g-2 \text{ arrows}"] \ar[loop,distance= 3em,in = 30, out=330,"2g \text{ loops}"'] & &
		\circled{2} \ar[loop, distance= 3em,in = 120, out=60,"2g \text{ loops}"']
	\end{tikzcd}
\end{equation*}

The numbers inside the vertices represent the corresponding dimension
vector determined by the object. All of the objects of class $v$
have the following Ext-quiver.
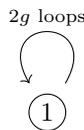
\begin{equation*}
	\begin{tikzcd}
		\circled 1  \ar[loop, distance= 3em,in = 120, out=60,"2g \text{ loops}"']
	\end{tikzcd}
\end{equation*}

\end{remark}

\begin{remark}
	[Isosingularity of the moduli problems]
	\label{rm-isosingularity}
	Each of the Settings~1-4 appear as examples in \cite[§7]{davison2021Purity2CalabiYauCategories}. 
	Thus by the \'etale Ext-quiver neighborhood theorem  \cite[Theorem~5.11]{davison2021Purity2CalabiYauCategories} and Remark~\ref{rm-ext-quivers}
we have that the morphisms
$\ModuliStack_2^{\sst} \to \ModuliSpace_2^{\sst}$ for each of the Settings~1-4 (and
$\ModuliStack_1^{\st} \to \ModuliSpace_1^{\st}$)
are pairwise \'etale locally isomorphic. Thus it suffices to check local properties for all settings by checking it for Setting~\ref{set-preproj}. This implies that the moduli spaces $\ModuliSpace_2$ (and $\ModuliSpace_1$) are pairwise stably isosingular for each of the Settings~1-4. See \cite[§2.4]{mauri2021IntersectionCohomologyRank} for a definition and further discussion of stable isosingularity.
\end{remark}

\subsection{The stratification}
The standard strategy to compute the motivic invariant of a space, such as the E-series of a stack, is to stratify the space into locally-closed pieces for which the E-series is known or easy to determine
and then add everything up by the cut-and-paste relation.
The calculation below is an execution of this strategy for $\ModuliStack_2^{\sst}$.

\subsubsection{The stratification of the good moduli space}

The points of the good moduli space
$\ModuliSpace^\sst_{2}$ correspond to
polystable objects of class $v$. 
We stratify $\ModuliSpace_2^{\sst}$ by polystability type.

First we distinguish between 
stable and strictly polystable objects of class $2v$.
Let $\ModuliSpace^\st_{2} \subset \ModuliSpace^\sst_{2}$ 
be the locus of stable objects of class $2v$.
Its complement 
$\Sigma = \ModuliSpace^{\sst}_{2} \setminus \ModuliSpace^{\st}_{2}$
is the locus of strictly polystable objects
\begin{equation*}
\Sigma = \{ L_1 \oplus L_2 
\mid L_1,L_2 %
\text{ objects of class $v$} 
\}.
\end{equation*} 
More precisely $\Sigma$ is the image of the direct sum map 
\begin{equation*}
	\oplus\colon \ModuliSpace^{\st}_1 \times \ModuliSpace^{\st}_1 \longto \ModuliSpace^{\st}_{2}
\end{equation*}
which is isomorphic to the symmetric square of
$\ModuliSpace_{1}^{\st}$
\begin{equation}
	\label{eq-iso_sigma-sym2}
	\Sigma \cong \Sym^2(\ModuliSpace_{1}^{\st}).
\end{equation}
Indeed, the direct sum map is a quasi-finite map onto a normal target (which follows from the normality in the case of preprojective algebras \cite{crawley-boevey2003normality} and Remark~\ref{rm-isosingularity}).
Thus by Zariski's Main Theorem the induced map $\Sym^2(\ModuliSpace^{\st}_1) \to \ModuliSpace^{\sst}_2$ is
is an isomorphism onto its image.
Similarly, by \cite[Theorem~3.2]{lebruyn2002noncommutativesmoothnesscoadjoint}, Remark \ref{rm-ext-quivers}, and Remark \ref{rm-isosingularity}, we deduce that 
$\Sigma$ is precisely the singular locus of
$\ModuliSpace^\sst_{2}$.

By the isomorphism \eqref{eq-iso_sigma-sym2}
the singular locus $\Omega$ of $\Sigma$ is identified with
the image of $\ModuliSpace_{1}^{\st}$ by the diagonal embedding  $\ModuliSpace_{1}^{\st} \into \Sym^2(\ModuliSpace_{1}^{\st})$.
Explicitly $\Omega$ is given by the locus of polystable  objects 
that are direct sums of two copies of the same object of class $v$
\begin{equation*}
\Omega = \{L^{\oplus 2} \mid
L\text{ object of class }v\} \subset \Sigma.
\end{equation*}

These loci yield the stratification by polystability type of
$\ModuliSpace^\sst_{2}$
\begin{equation}
\label{eq-crs_strat}
\ModuliSpace^\sst_{2} = \ModuliSpace^\st_{2} 
\cup (\Sigma \setminus \Omega ) \cup \Omega,
\end{equation}
where the stratum $\Sigma \setminus \Omega$ has the explicit description
\begin{equation*}
\Sigma \setminus \Omega = \{ L_1 \oplus L_2 
\mid L_1,L_2 \text{ distinct objects of class }v\}.
\end{equation*}

This stratification has already appeared in the literature 
and is applied in the works
\cite{felisetti2021IntersectionCohomologyModuli}
\cite{mauri2021IntersectionCohomologyRank} 
to compute the intersection E-polynomials of $\ModuliSpace^\sst_{2}$.

\begin{remark}
The three loci $\ModuliSpace_{2}^{\st}$, $\Sigma\setminus \Omega$, and $\Omega$ correspond, in order, to the first three Ext-quivers in Remark~\ref{rm-ext-quivers}. We emphasize that the deepest stratum $\Omega$ corresponds to the $g$-loop quiver with dimension vector 2.
\end{remark}

\subsubsection{Pulling back the stratification} A natural stratification of the stack $\ModuliStack^{\sst}_{2}$ is 
the pull-back of the stratification \eqref{eq-crs_strat} along the morphism 
$\coarsemap=\coarsemap_{2}\colon 
\ModuliStack^\sst_{2} \to \ModuliSpace^\sst_{2}$
to the good moduli space.
Write
\begin{equation*}
\SingStack = \coarsemap^{-1}(\SingSpace),
\DiagStack = \coarsemap^{-1}(\DiagSpace),
\text{ and }
\OffDiagStack = \coarsemap^{-1}(\SingSpace \setminus \DiagSpace) 
= \SingStack \setminus \DiagStack.
\end{equation*}
The stack $\SingStack$ is the singular locus of $\ModuliStack_{2}^\sst$.
We call $\OffDiagStack$ the \emph{off-diagonal locus} 
and $\DiagStack$ the \emph{diagonal locus}.
Additionally, we have the stable locus
$\ModuliStack^{\st}_{2} = p^{-1}(\ModuliSpace^\st_{2}) 
$, which is a $\GG_m$-gerbe over $\ModuliSpace^{\st}_2$.

The pullback stratification is 
\begin{equation*}
\ModuliStack^{\sst}_{2} = \ModuliStack^{\st}_{2} \cup
\OffDiagStack \cup \DiagStack.
\end{equation*}

The E-series of the stable loci $\ModuliStack^{\st}_1$ and $\ModuliStack^{\st}_{2}$ of the stacks 
is calculated  from the E-polynomial of of the stable loci of the $\ModuliSpace^{\st}_{1}$ and $\ModuliSpace^\st_{2}$ of the good moduli spaces.
\begin{lemma}
We have 
\begin{equation*}
		E(\ModuliStack^\st) = E(\ModuliSpace^\st/\GG_m) 
	= E(\ModuliSpace^\st)/(\EGm) 
\end{equation*}
\end{lemma}
\begin{proof}
	We apply \cite[Lemma~3]{heinloth2012cohomology} to the $\GG_m$-gerbe
	$\ModuliStack^{\st} \to \ModuliSpace^{\st}$.
	Thus it suffices to construct a vector bundle on $\ModuliStack^{\st}$ of $\GG_m$-weight 1.

	In Settings 2 and 4, we take the tautological bundle $\CE$ on $\ModuliStack^{\st}$ which records the underlying vector space of the representation. The $\GG_m$-weight of $\CE$ is given by the weight of the scaling action which is equal to 1.

	For Settings 1 and 3 see
	\cite[Proposition~4.6.2]{huybrechts2010geometrymodulispaces} and its proof, which applies to Setting~\ref{set-Higgs} by the BNR correspondence \cite{beauville1989spectralcurvesgeneralized}.
\end{proof}

Therefore, to compute 
$E(\ModuliStack_{2}^\sst)$
it remains to compute the E-polynomials of the strata
$\OffDiagStack$ and $\DiagStack$.

\subsubsection{Stratification of the strictly semistable locus}

For every strictly semistable object $E$ of class $2v$ 
there exists stable objects $K$ and $Q$ of class $v$ and a short exact sequence
\begin{equation*}
\begin{tikzcd}
W(E)\colon &
0 \ar[r] &
K \ar[r] &
E \ar[r] &
Q \ar[r] &
0
\end{tikzcd}
\end{equation*}
that witnesses the strict semistability of  $E$. We call $W(E)$ 
the \emph{semistabilizing} short exact sequence.

There are at most two isomorphism classes of 
objects of class $v$ 
that can appear as the subobject or quotient 
in a semistabilizing short exact sequence. 

If $W(E)$ is non-split,
then the semistabilizing short exact sequence $W(E)$ is unique 
up to (non-unique) isomorphism.
On the other hand, if $W(E)$ is split, that is, if $E$ is polystable,
then all semistablizing short exact sequences are split.

Altogether, using the short exact sequences $W(E)$ we can 
distinguish four types of strictly semistables
\begin{itemize}
\item $W(E)$ is non-split, $K \ncong Q$
\item $W(E)$ non-split, $K \cong Q$
\item $W(E) $ split, $K \ncong Q$
\item $W(E)$ split, $K \cong Q$
\end{itemize}
We stratify the moduli stack $\ModuliStack^\sst_2$ according to these four cases.

Let $\SingSplit \subset \SingStack$ be the image of the direct-sum morphism
\begin{equation*}
\begin{split}
s\colon \ModuliStack_1^\st \times \ModuliStack_1^\st &\longto \SingStack\\
(K,Q) &\longmapsto K \oplus Q.
\end{split}
\end{equation*}
This locus parametrizes those strictly semistables
admitting a split semistabilizing short exact sequence. Denote the complement of $\SingSplit$ in $\SingStack$
by $\SingNonSplit$.
The stack $\SingNonSplit$ parameterizes 
those strictly semistable objects
with non-split semistabilizing short exact sequence. 

These are two new loci that we intersect with the strata 
$\OffDiagStack$ and $\DiagStack$ to obtain our final stratification.
Set
\begin{equation*}
\begin{tikzcd}[row sep = tiny]
\OffDiagNonSplit = \OffDiagStack \cap \SingNonSplit, 
& \DiagNonSplit = \DiagStack \cap \SingNonSplit,\\
\OffDiagSplit = \OffDiagStack \cap \SingSplit,
&\DiagSplit = \DiagStack \cap \SingSplit.
\end{tikzcd}
\end{equation*}
This defines the stratification 
\begin{equation*}%
\ModuliStack_{2}^\sst = \ModuliStack_{2}^\st \sqcup \OffDiagNonSplit 
\sqcup \OffDiagSplit \sqcup \DiagNonSplit \sqcup \DiagSplit
\end{equation*}
that we ultimately use to compute the E-series of $\ModuliStack_{2}$.

\subsection{The stack of strictly semistables and the stack of short exact sequences}
Let $\ExactStack$ be the stack of short exact sequences
\begin{equation*}
\begin{tikzcd}
0 \ar[r]&
K \ar[r]&
E \ar[r]&
Q \ar[r]&
0
\end{tikzcd}
\end{equation*}
where $K,Q$ are stable objects of class $v$.
The following convolution diagram relates $\ExactStack$ to the
moduli stacks $\ModuliStack_1^\st$ and $\ModuliStack_2^{\sst}$.
\begin{equation*}
\begin{tikzcd}
&\ExactStack \ar[dl,"\epsilon"'] \ar[dr,"\pi"]& &\\
\ModuliStack_{1}^{\st} \times \ModuliStack_{1}^{\st}& & \SingStack \ar[r,hook]& \ModuliStack_{2}^\sst 
\end{tikzcd}
\end{equation*}
In the diagram the morphism 
$\epsilon\colon \ExactStack \to \ModuliStack_{1}^{\st}$
map a short exact sequence to the pair consisting of the quotient object and the subobject
\begin{equation*}
\epsilon(K\into E\onto Q) = (Q,K).
\end{equation*}
The morphism 
$\pi\colon \ExactStack \to \ModuliStack_{2}^\sst$
maps a short exact sequence  to the middle term
\begin{equation*}
\pi(K \into E \onto Q) = E.
\end{equation*}
An extension of two objects of class $v$ is
necessarily strictly semistable,
thus 
the morphism $\pi$ indeed factors through $\SingStack$.

Let
$\ExactSplit \subset \ExactStack$  be
the closed substack of split short exact sequences 
and let $\ExactNonSplit \subset \ExactStack$ be its complement, which
is the open substack of 
non-split short exact sequences. 
To compute the E-series of $\ExactSplit$ and $\ExactNonSplit$ we identify $\ExactStack$ as a Picard stack over 
$(\ModuliStack_{1}^\st)^{\times 2}$ and $\ExactSplit$ as its zero-section.

\subsubsection*{Aside on Picard stacks}
\label{subsec-picard_stacks}
For the convenience of the reader we recall the 
the definition and computation of some E-series of associated to a Picard stack.
For more details see
\cite[Expos\'e XVIII, Section~1.4]{SGA4}
and \cite[(14.4),(14.5)]{laumon2000ChampsAlgebriques}.

Let $\FB$ be a finite type Artin stack with affine stabilizers 
and let $\mathcal{F}$ be a coherent sheaf on $\mathfrak{B}$. 
Recall that the \emph{total space of $\mathcal{F}$} is the
relative spectrum of the symmetric algebra $\Sym_{\FB}(\mathcal{F}^{\vee})$, 
\begin{equation*}
\Tot_{\FB}(\mathcal{F}) =
\relSpec_{\FB}(\Sym_{\FB}(\mathcal{F}^\vee)) \longto \FB.
\end{equation*}

Let
$\mathcal{F}^{\bullet}=\mathcal{F}^{-1} \stackrel{d}{\to}\mathcal{F}^{0}$
be a two-term complex of coherent sheaves on $\FB$. 
Define the \emph{Picard stack} $\Tot_{\FB}(\mathcal{F}^{\bullet})$ 
associated to the two-term complex $\mathcal{F}^{\bullet}$ explicitly as 
follows. 
For every morphism $u\colon U \to \FB$ from an affine scheme $U$ we define
\begin{equation*}
\Tot_{\FB}(\mathcal{F}^\bullet)(U) = 
\begin{Bmatrix}
\text{objects}= H^{0}(U,u^\ast \mathcal{F}^{\bullet})(U) 
\\
\text{morphisms} = 
 H^{-1}(U,u^\ast\mathcal{F}^{\bullet})(U) 
\end{Bmatrix}.
\end{equation*}
By interpreting $d\colon \mathcal{F}^{-1} \to \mathcal{F}^{0}$ as an action
of the group stack $\Tot_{\FB}(\mathcal{F}^{-1})$ on the stack 
$\Tot_{\FB}(\mathcal{F}^{0})$ we have the description
\begin{equation*}
\Tot_{\FB}(\mathcal{F}^{\bullet}) =
\Tot_{\FB}(\mathcal{F}^{0})/\Tot_{\FB}(\mathcal{F}^{-1}).
\end{equation*}
A chain map $\CF^{\bullet} \to \CG^{\bullet}$ of two-term complexes of coherent sheaves which is a quasi-isomorphism induces an isomorphism of Picard stacks
$\Tot_{\FB}(\CF^{\bullet}) \isoto \Tot_{\FB}(\CG)^{\bullet}$.

In general the Picard stack $\Tot_{\FB}(\mathcal{F}^\bullet)$
need not be an Artin stack.
However, if $\mathcal{F}^{-1}$ is locally free, then $\Tot_{\FB}(\CF^\bullet)$ is an Artin stack with affine stabilizers 
(\cite[page 143]{laumon2000ChampsAlgebriques}).

The \emph{zero-section} of the Picard stack $\Tot_{\FB}(\CF^{\bullet})$
is the closed immersion of Picard stacks 
\begin{equation*}
\Tot_{\FB}(\ker(d)[1])=
\Tot_{\FB}(\tau^{\leq -1}\CF^{\bullet}) \longinto
\Tot_{\FB}(\CF^\bullet).
\end{equation*}
where $\tau^{\leq i}$ denotes the standard truncation

\begin{lemma}
	\label{lm-motive_zero_sec_picard_perfect}
	Suppose $\CF^{\bullet} = \CF^{-1} \to \CF^{0}$ is a two-term complex of coherent sheaves which is quasi-isomorphic to a complex of locally free sheaves with amplitude non-positive degrees.
	Then 
	\begin{equation*}
		[\Tot_{\FB}(\CF^{\bullet})] = [\FB] q^{\chi(\mathcal{F})}
	\end{equation*}
\end{lemma}
If $\FB$ has the resolution property, then every two-term complex of coherent sheaves satisfies the assumption.
See \cite{thomason1987equivariantresolutionlinearization,totaro2008resolution,gross2017tensorgeneratorsschemes} for general criteria for stacks to have the resolution property.
\begin{proof}\footnote{Thank you to the referee for suggesting this argument.}
	Up to stratifying with respect to an open cover, we can assume
	without loss of generality that there is a resolution $\CE^\bullet \to \CF^\bullet$
	such that $\CE^i \cong \CO_{\FB}^{n_i}$ are trivial vector bundles and $\CE^{i} = 0$ for $i > 0$.
	Taking inspiration from the proof of \cite[Lemma~3.3]{garcia-prada2014MotivesModuliChains}
	we stratify further along loci $Z_{r_{1},r_{2}}$ for which the differential 
	$d_{\CE}^{-i}$ of $\CE^\bullet$ has constant rank $r_i$ for $i =1,2$.
	Along the strata $Z_{r_{1},r_{2}}$, the truncation $\tau^{\geq -1}\CE$,
	which is quasi-isomorphic to $\CE^\bullet$,
	is a two-term complex of vector bundles.
	After applying \cite{garcia-prada2014MotivesModuliChains} to this complex,
	we deduce the result by consolidating the stratifications.
\end{proof}

We compute the E-polynomial of the complement of the zero-section via the cut-and-paste relation.
\medskip

Let $(\CQ,\CK)$ be the tautological pair of objects over 
$(\ModuliStack_1^{\st})^{\times 2}$.
We have the tautological Hom-sheaf $\SHom(\CQ,\CK)$
on $(\ModuliStack_{1}^{\st})^{\times 2}$, defined as follows.
For every morphism $t\colon U \to (\ModuliStack_1^{\st})^{\times 2}$ 
out of an affine scheme $U$
the coherent sheaf $t^{\ast}\SHom(\CQ,\CK)$ is defined to be the 
sheaf associated to the coherent sheaf 
$\SHom_{\sC_U}(t^{\ast}\CQ,t^{\ast}\CK)$ on $U$. 
We define the complex of coherent sheaves $\R \SHom (\CQ,\CK)$ similarly.

We give a few more details for how to construct these complexes in each of the Settings~1-4.
In Settings \ref{set-Higgs} and \ref{set-cpt2CY} we have the universal object
$\CE$ living over $\ModuliStack_1^{\st} \times X$ where $X = C$ in the case of Setting~\ref{set-Higgs} or $X=S$ in the case of Setting~\ref{set-cpt2CY}.
Define the tautological objects as pullbacks $\CQ\coloneqq \pr_{13}^{*}\CE$ and $\CK \coloneqq \pr_{23}^{*}\CE$, where $\pr_{23},\pr_{13}\colon \ModuliStack_{1}^{\st} \times \ModuliStack_{1}^{\st} \times X \to \ModuliStack_{1}^{\st} \times X$ are the projections.
Then the complex $\R\SHom(\CQ,\CK)$ is the complex $(\pr_{12})_{*}\R\SHom_{(\ModuliStack_{1}^{\st}\times X)^{\times 2}}(\CQ,\CK)$.

In Settings \ref{set-pi_1-mod} and \ref{set-preproj} the moduli stacks $\ModuliStack_{1}^{\st}$ parametrise representations of an algebra $A$: $\CC[\pi_1(\Sigma_g \setminus p)]$ in Setting~\ref{set-pi_1-mod} and $\Pi_{S_g}$ in Setting~\ref{set-preproj}. ~
We consider the tautological vector bundle $\CV$ over $\ModuliStack_1^{\st}$. Over a point of $\ModuliStack_1^{\st}$ corresponding to a representation $\rho$, $\CV_{[\rho]}$ is the underlying vector space of the representation.
There is a morphism of algebras $A \to \End_{\ModuliStack_1^{\st}}(\CV)$ which endows $\CV$ with the structure of a $A\otimes \CO_\ModuliStack$-module. 
Pulling back to $\ModuliStack_1^{\st} \times \ModuliStack_1^{\st}$ we obtain the tautological objects $\CQ \coloneqq \pr_{1}^{*}\CV$ and $\CK \coloneqq \pr_2^*\CV$ which are $A \otimes \CO_{(\ModuliStack_1^{\st})^{\times 2}}$-modules.
The complex $\R\SHom(\CQ,\CK)$ is the complex $\R\SHom_{A \otimes \CO_{(\ModuliStack_{1}^{\sst})^{\times 2}}}(\CQ,\CK)$.

\begin{lemma}
\label{lm-exact_picard_stack}
The edge term morphism 
$\epsilon\colon \ExactStack \to (\ModuliStack_{1}^{\st})^{\times 2}$
is isomorphic to the Picard stack over
$(\ModuliStack_{1}^{\st})^{\times 2}$
associated to the two-term complex
\begin{equation}
\label{eq-ext_complex}
\tRHom = \tau^{\leq 0}\R \SHom(\CQ,\CK)[1].
\end{equation}
Under this isomorphism $\ExactSplit \into \ExactStack$ is the zero-section and $\ExactNonSplit$ is its complement.
\end{lemma}
\begin{proof}
See (the proof of) \cite[Proposition~2.3.4]{kapranov2022CohomologicalHallAlgebra}
\end{proof}

\subsection{The locus of non-split semistablizing short exact sequences}
\label{sec-nonsplit_ses}

Over the non-split locus $\SingNonSplit$ 
the middle-term morphism $\pi$ is an isomorphism.
\begin{lemma}
\label{lm-iso_nonsplit}
The morphism of stacks mapping a strictly semistable non-polystable object to its
semistabilizing short exact sequence
\begin{equation*}
\begin{split}
W\colon \SingNonSplit &\longto  \ExactNonSplit\\
E &\longmapsto W(E)
\end{split}
\end{equation*}
is an isomorphism with inverse given by the projection to the middle term 
$\pi\colon \ExactStack \to \SingStack$.
\end{lemma}

Consider the four Cartesian squares
\begin{equation*}
\begin{tikzcd}
\ExactStack_\NonIsoPairs \ar[r,hook] \ar[d,"\epsilon_{\NonIsoPairs}"]& 
\ExactStack \ar[d,"\epsilon"]
& \ar[l,hook'] \ExactStack_\IsoPairs \ar[d,"\epsilon_{\IsoPairs}"'] 
\\
\NonIsoPairs 
\ar[d] \ar[r,hook]& 
(\ModuliStack_{1}^{\st})^{\times 2} \ar[d]&
\ar[l,hook'] \IsoPairs \ar[d] %
\\
(\ModuliSpace_{1}^{\st})^{\times 2} \setminus \Delta(\ModuliSpace_{1}^{\st})
\ar[r,hook]& 
(\ModuliSpace_{1}^{\st})^{\times 2} &
\ar[l,"\Delta"',hook'] \ModuliSpace_{1}^{\st},
\end{tikzcd}
\end{equation*}
where $\Delta\colon \ModuliSpace_{1}^{\st} \to (\ModuliSpace_{1}^{\st})^{\times 2}$
is the diagonal. 
The stack $\NonIsoPairs$ parametrizes pairs of non-isomorphic stable objects of class $v$.
The stack $\IsoPairs$ parametrizes pairs of isomorphic 
stable objects of class $v$.

There are isomorphisms of stacks
\begin{equation*}
\begin{split}
\NonIsoPairs &\cong ((\ModuliSpace_{1}^{\st} \times \ModuliSpace_{1}^{\st} ) 
\setminus \Delta(\ModuliSpace_{1}^{\st}))/\GG_m^2, \\
\IsoPairs &\cong \ModuliSpace_{1}^{\st} /\GG_m^2.
\end{split}
\end{equation*}
Thus their E-series are
\begin{equation*}
\begin{split}
E(\NonIsoPairs) &= \frac{E(\ModuliSpace_1^{\st})^2 -E(\ModuliSpace_1^{\st})}{(\EGm)^2},
\\
E(\IsoPairs) &= \frac{E(\ModuliSpace_1^{\st})}{(\EGm)^2}.
\end{split}
\end{equation*}

By Lemmas~\ref{lm-iso_nonsplit}~and~\ref{lm-exact_picard_stack} we have 
for the non-split loci of the stack of short exact sequences
\begin{equation*}
\begin{split}
\OffDiagNonSplit &\cong \ExactNonSplit_{\NonIsoPairs} 
\cong \ExactStack_{\NonIsoPairs} \setminus 
(\ExactStack_\NonIsoPairs \cap \ExactStack_0),
\\
\DiagNonSplit &\cong \ExactNonSplit_{\IsoPairs}
\cong \ExactStack_{\IsoPairs} \setminus
(\ExactStack_\IsoPairs \cap \ExactStack_0).
\end{split}
\end{equation*}
Thus to compute the E-series of $\OffDiagNonSplit$ and $\DiagNonSplit$
it remains to compute the E-series of $\ExactNonSplit_{\NonIsoPairs}$
and $\ExactNonSplit_{\IsoPairs}$.

We apply 
Lemma~\ref{lm-motive_zero_sec_picard_perfect} to the 
 restrictions of $\tRHom$ to the loci $\IsoPairs$ and $\NonIsoPairs$
to compute the E-series of $\ExactNonSplit_\IsoPairs$ and 
$\ExactNonSplit_\NonIsoPairs$.

\begin{lemma}
The restriction of the complex
$\tRHom$
to $\NonIsoPairs$ is quasi-isomorphic to a 
rank $2g-2$ vector bundle supported in degree 0. Hence
\begin{equation*}
E(\OffDiagNonSplit) = \frac{(E(\ModuliSpace_1^{\st})^2 - E(\ModuliSpace_1^{\st}))
(\ELL^{2g-2}-1)}{(\EGm)^2} .
\end{equation*}
\end{lemma}
\begin{proof}
For all non-isomorphic stable objects of class $v$
$L_1,L_2$ 
we have 
\begin{equation*}
\Hom(L_1,L_2) = 0 
\text{ and } \dim \Ext^1(L_1,L_2) = 2g-2.
\end{equation*}
Thus over $\NonIsoPairs$, the complex $\tRHom$ is  
concentrated in degree 0.
The rank of the degree 0 term is constant and equal to $2g-2$.
Since $\ModuliStack_{1}^{\st}$ is smooth we deduce the lemma.
\end{proof}

\begin{lemma}
The degree $-1$ cohomology of the restriction of the complex
$\tRHom$
to $\IsoPairs$ is a rank $1$ vector bundle and the degree $0$ cohomology is a rank $2g$ vector bundle. 
Hence
\begin{equation*}
E(\DiagNonSplit) = \frac{E(\ModuliSpace_1^{\st}) (\ELL^{2g}-1)}{\ELL(\EGm)^2}.
\end{equation*}
\end{lemma}
\begin{proof}
For all stable objects $L$ of class $v$ we have
\begin{equation*}
\Hom(L,L) \cong \CC
\text{ and } \dim \Ext^1(L,L) = 2g.
\end{equation*}
Thus over $\IsoPairs$, the cohomology sheaf of 
the complex  $\tRHom$ in 
degree $-1$ has constant rank equal to $1$ and
in degree $0$ has constant rank equal to $2g$.
For every Setting 1-4, the stacks $\ModuliStack_1^{s}$ have the resolution property.
Thus we can apply Lemma~\ref{lm-motive_zero_sec_picard_perfect} to deduce the identity for $E(\DiagNonSplit)$.
\end{proof}

\begin{corollary}
\label{cor-nonsplitss}
By the cut-and-paste relation for $\SingNonSplit = \OffDiagNonSplit \sqcup \DiagNonSplit$ we have
\begin{equation*}
E(\SingNonSplit) 
=\frac{E(\ModuliSpace_1^{\st})^2}{(\EGm)^2}(\ELL^{2g-2}-1)
+ \frac{E(\ModuliSpace_1^{\st})}{\ELL (\EGm)}(\ELL^{2g-1}+1).
\end{equation*}
\end{corollary}

\subsection{The locus of polystables}

Over the split off-diagonal locus $\OffDiagSplit$ the morphism 
$\pi$ is a double cover: 
a direct sum $K\oplus Q$ with non-isomorphic summands,
arises as the middle term of precisely two different 
(isomorphism classes of) short exact sequences with edge terms 
stable objects of class $v$.
\begin{lemma}
There is a commutative diagram with horizontal maps isomorphisms
\begin{equation*}
\begin{tikzcd}
((\ModuliSpace_{1}^{\st})^{\times 2} \setminus \Delta(\ModuliSpace_{1}^{\st}))/\GG_m^2
\ar[r,"\cong"',"h"] \ar[d]& 
\pi^{-1}(\OffDiagSplit) \ar[d,"\pi"] \\
(((\ModuliSpace_{1}^{\st})^{\times 2} \setminus \Delta(\ModuliSpace_{1}^{\st}))
/\GG_m^2 )/(\ZZ/2)
\ar[r,"\cong"',"\tilde{h}"]
&
\OffDiagSplit
\end{tikzcd}
\end{equation*}
where the $\ZZ/2$-action on 
$((\ModuliSpace_{1}^{\st})^{\times 2} \setminus
\Delta(\ModuliSpace_{1}^{\st}))/\GG^2_m$ 
is induced by the $\ZZ/2$-actions on 
$(\ModuliSpace_{1}^{\st})^{\times 2}$ and $\GG_m^2$
given by swapping the two factors.
Thus
\begin{equation*}
E(\OffDiagSplit) = \frac{\ELL E(\Sym^2(\ModuliSpace_{1}^{\st})) + E(\MAlt^2[\ModuliSpace_{1}^{\st}])}{(\EGm)^2(\EPP)} - \frac{\ELL E(\ModuliSpace_{1}^{\st})}{(\EGm)^2(\EPP)}.
\end{equation*}
\end{lemma}
\begin{proof}
The stack $\pi^{-1}(\OffDiagSplit)$ is the stack of split short exact sequences
\begin{equation*}
\begin{tikzcd}
0 \ar[r] & K \ar[r]& E \ar[r] & Q \ar[r] & 0 
\end{tikzcd}
\end{equation*}
such that 
$K$ and $Q$ are non-isomorphic stable objects of class $v$. 
The morphism $h$ maps a pair 
$(Q,K)$ to the short exact sequence
\begin{equation*}
\begin{tikzcd}
0 \ar[r] & K \ar[r]& K \oplus Q\ar[r] 
& Q \ar[r] & 0.
\end{tikzcd}
\end{equation*}
A (quasi-)inverse is induced from the projection $\pi^{-1}(\OffDiagSplit) \to \NonIsoPairs \to (\ModuliSpace_1^{\st}) \setminus \Delta(\ModuliSpace_1^{\st})$. 
There is a $\ZZ/2$-action on the stack $\pi^{-1}(\OffDiagSplit)$ which swaps the roles of subobject and quotient in the split short exact sequence and $\OffDiagSplit \cong \pi^{-1}(\OffDiagSplit)/(\ZZ/2)$.
Under the isomorphism $h$ this agrees with the $\ZZ/2$-action on $((\ModuliSpace_{1}^{\st})^{\times 2} \setminus
\Delta(\ModuliSpace_{1}^{\st}))/\GG^2_m$. Thus we have the desired diagram.

By the cut and paste relation we have 
\begin{equation*}
E((((\ModuliSpace_{1}^{\st})^{\times 2} \setminus 
\Delta(\ModuliSpace_{1}^{\st}))/\GG_m^2)/(\ZZ/2)) = 
E(\Sym^2(\ModuliSpace_{1}^{\st}/\GG_m)) - 
E((\Delta(\ModuliSpace_{1}^{\st})/\GG_m^2)/(\ZZ/2)).
\end{equation*}
Note that both the $\GG_m^2$-action and $\ZZ/2$-action on 
$\Delta(\ModuliSpace_{1}^{\st})$ are trivial, hence
\begin{equation*}
(\Delta(\ModuliSpace_{1}^{\st})/\GG_m^2)/(\ZZ/2) \cong 
\ModuliSpace_{1}^{\st} \times (B\GG_m)^2/\ZZ/2
\cong  
\ModuliSpace_{1}^{\st} \times \Sym^2(B\GG_m).
\end{equation*} 
The required expression for $E(\OffDiagSplit)$ follows from Lemma~\ref{lm-epoly_quot_stacks}.
\end{proof}

Over $\DiagSplit$ the morphism $\pi$ is the 
Zariski-locally trivial $\PP^1$-fibration 
given by the the quotient morphism 
$\ModuliSpace_{1}^{\st}/B \to \ModuliSpace_{1}^{\st}/\GL_2$,
where $B \subset \GL_2$ 
is the subgroup of upper-triangular matrices. 
\begin{lemma}
The split diagonal locus $\DiagSplit$ is isomorphic to
the quotient stack $\ModuliSpace_{1}^{\st}/\GL_2$. Thus
\begin{equation*}
E(\DiagSplit) = \frac{E(\ModuliSpace_1^\st)}{\EGLtwo}.
\end{equation*}
\end{lemma}

\begin{corollary}
\label{cor-splitss}
By the cut-and-paste relation for $\SingSplit = \OffDiagSplit \cup \DiagSplit$ we have
\begin{equation*}
E(\SingSplit)=\frac{\ELL E(\Sym^2(\ModuliSpace_{1}^{\st})) + E(\Lambda^2[\ModuliSpace_{1}^{\st}])}{(\EGm)^2(\EPP)}  
- \frac{E(\ModuliSpace_1^{\st})}{\ELL (\EGm)} .
\end{equation*}
\end{corollary}

\subsection{Adding it all up}

\begin{theorem}
\label{thm-main_calc}
We have
\begin{equation*}
\begin{split}
E(\ModuliStack_{2}^\sst) 
&=E(\ModuliStack_{2}^\st) 
+ \frac{\ELL^{2g-1} +\ELL^{2g-2}-1}{(\EGm)^2(\EPP)}
E(\Sym^2(\ModuliSpace_{1}^{\st})) \\
&\qquad + \frac{\ELL^{2g-1}+\ELL^{2g-2}-\ELL}{(\EGm)^2(\EPP)}
E(\Lambda^2[\ModuliSpace_{1}^{\st}])
 + \frac{\ELL^{2g-2}E(\ModuliSpace_{1}^{\st})}{(\EGm)}.
\end{split}
\end{equation*}

Thus if the cohomological integrality conjecture (Definition~\ref{def-cohintconj}) holds
\begin{equation}
\label{eq-e_stack_minus_sym2}
\begin{split}
E(\BPS_2) 
&=\frac{E(\ModuliSpace_{2}^\st)}{\ELL^{4g-3}} 
+ \frac{E(\PP^{2g-3})}{\ELL^{4g-3}(\EPP)}
E(\Sym^2(\ModuliSpace_{1}^{\st})) \\
&\qquad 
+ \frac{E(\PP^{2g-3})}{\ELL^{4g-4}(\EPP)}
E(\Lambda^2[\ModuliSpace_{1}^{\st}])\\
&\qquad + \frac{E(\ModuliSpace_{1}^{\st})}{\ELL^{2g-1}}.
\end{split}
\end{equation}
\end{theorem}
\begin{proof}
First using Corollaries~\ref{cor-nonsplitss}~and~\ref{cor-splitss} we gather the terms which are linear in $E(\ModuliSpace_{1}^{\st})$.

\begin{equation*}
\begin{split}
E(\SingStack) &= E(\SingNonSplit) + E(\SingSplit) \\
&=\frac{E(\ModuliSpace_1^{\st})^2}{(\EGm)}E(\PP^{2g-3}) 
+ \frac{\ELL E(\Sym^2(\ModuliSpace_{1}^{\st})) + E(\Lambda^2[\ModuliSpace_{1}^{\st}])}{(\EGm)^2 (\EPP)}+ \frac{\ELL^{2g-2}E(\ModuliSpace_1^{\st})}{(\EGm)}
\end{split}
\end{equation*}

Using 
$E(\ModuliSpace_1^{\st})^2 = 
E(\Sym^2(\ModuliSpace_{1}^{\st})) + E(\Lambda^2[\ModuliSpace_{1}^{\st}])$
we have 
\begin{equation*}
\begin{split}
E(\SingStack)&=\frac{E(\ModuliSpace_1^{\st})^2}{(\EGm)}E(\PP^{2g-3}) 
+ \frac{\ELL E(\Sym^2(\ModuliSpace_{1}^{\st})) + E(\Lambda^2[\ModuliSpace_{1}^{\st}])}{(\EGm)^2(\EPP)} + \frac{\ELL^{2g-2}E(\ModuliSpace_1^{\st})}{(\EGm)}\\
&= \frac{\ELL^{2g-1} + \ELL^{2g-2}-1}{(\EGm)^2(\EPP)}
E(\Sym^2(\ModuliSpace_{1}^{\st})) 
+ \frac{\ELL^{2g-1}+\ELL^{2g-2}-\ELL}{(\EGm)^2(\EPP)}
E(\Lambda^2[\ModuliSpace_1^{\st}]) \\&\qquad+ \frac{\ELL^{2g-2}E(\ModuliSpace_1^{\st})}{(\EGm)}.
\end{split}
\end{equation*}
By the cut-and-paste relation 
$ E(\ModuliStack_{2}^\sst) = 
E(\ModuliStack_{2}^\st) + E(\SingStack)$
we have
\begin{equation*}
\begin{split}
E(\ModuliStack_{2}^\sst) 
&=E(\ModuliStack_{2}^\st) 
+ \frac{\ELL^{2g-1} +\ELL^{2g-2}-1}{(\EGm)^2(\EPP)}
E(\Sym^2(\ModuliSpace_1^{\st})) \\
&\qquad + \frac{\ELL^{2g-1}+\ELL^{2g-2}-\ELL}{(\EGm)^2(\EPP)}
E(\Lambda^2[\ModuliSpace_1^{\st}])
 + \frac{\ELL^{2g-2}E(\ModuliSpace_1^{\st})}{(\EGm)}.
\end{split}
\end{equation*}
Multiply $E(\ModuliStack_2^{\sst})$ by $\ELL^{4(1-g)}$ 
and subtract
\begin{equation*}
\begin{multlined}
E(\Sym^2(\H_c(\ModuliSpace_{1}^{\st}) 
\otimes \LL^{-g} \otimes \Hc^{\bullet}(B\GG_m) \otimes \LL)) 
= \frac{E(\Sym^2(\ModuliSpace_1^{\st}))}
{\ELL^{2g-3}(\EGm)^2(\EPP)} 
+ \frac{E(\Lambda^2[\ModuliSpace_1^{\st}])}
{\ELL^{2g-2}(\EGm)^2(\EPP)}
\end{multlined}
\end{equation*}
to deduce the required expression for $\frac{q}{q-1}E(\BPS_{2})$.
\end{proof}

\subsection{Comparison to intersection cohomology}

We recall Mauri's computation of the intersection cohomology of the coarse moduli spaces $\ModuliSpace_2^{\sst}$.
\begin{proposition}[{\cite[Theorem~1.3]{mauri2021IntersectionCohomologyRank}}]
	\label{thm-Mauri-computation}
	\begin{equation*}
		IE(\ModuliSpace_2^{\sst}) = E(\ModuliSpace_2^{\sst}) + 
		\frac{\ELL^{2g-4}-1}{\ELL^2-1}\left(\ELL^2 E(\Sym^2(\ModuliSpace_{1}^{\st})) +
		\ELL E(\Lambda^2[\ModuliSpace_1^{\st}])\right) + \ELL^{2g-2}E(\ModuliSpace_{1}^{\st})
	\end{equation*}
\end{proposition}
\begin{proof}
	\cite[Theorem~1.3]{mauri2021IntersectionCohomologyRank} is applicable by the stable isosingulariy of the moduli spaces $\ModuliSpace_2^{\sst}$ (Remark~\ref{rm-isosingularity})

The form we give here follows from the identities
\begin{equation*}
	\begin{split}
		E(\Sigma_{\iota})^{+} &= E(\Sym^2(\ModuliSpace_1^{\st}))\\
		E(\Sigma_{\iota})^{-} &= E(\Lambda^2[\ModuliSpace_1^{\st}])
	\end{split}
	\end{equation*}
	which themselves are deduced by considering the ramified double cover \begin{equation*}
		\Sigma_{\iota} = \ModuliSpace_1^{\st} \times \ModuliSpace_1^{\st} \to \Sym^2(\ModuliSpace_1^{\st}) \cong \Sigma.
	\end{equation*}
\end{proof}

\begin{proof}[Proof of Theorem~\ref{thm-general_e-poly}]
	By Proposition~\ref{thm-Mauri-computation} we have
\begin{equation*}
\begin{split}
IE&(\ModuliSpace_2^{\sst}) = E(\ModuliSpace_2^{\sst}) + 
\frac{\ELL^{2g-4}-1}{\ELL^2-1}\left(\ELL^2 E(\Sym^2(\ModuliSpace_{1}^{\st})) +
\ELL E(\Lambda^2[\ModuliSpace_1^{\st}])\right) + \ELL^{2g-2}E(\ModuliSpace_{1}^{\st})\\
&=
E(\ModuliSpace_2^{\st}) 
+ \frac{\ELL^{2g-2}-1}{\ELL^2-1}E(\Sym^2(\ModuliSpace_{1}^{\st}))
+\frac{q(\ELL^{2g-4}-1)}{\ELL^2-1}E(\Lambda^2[\ModuliSpace_1^{\st}]) + \ELL^{2g-2}E(\ModuliSpace_{1}^{\st})
\end{split}
\end{equation*}
where in the second line we use $E(\ModuliSpace_{2}^{\sst}) = E(\ModuliSpace_{2}^{\st}) + E(\Sym^2(\ModuliSpace_{1}^{\st}))$.

We shift by $\LL^{4(g-1)+1}$ (i.e.\ divide by $q^{4g-3}$) and add 
$E(\Lambda^2([\ModuliSpace_{1}^{\st}]\otimes \LL^{-g})) = q^{-2g} E(\Lambda^2[\ModuliSpace_1^{\st}])$
to obtain
\begin{equation}
\label{eq-IHeqBPS}
\begin{split}
\frac{IE(\ModuliSpace_{2}^{\sst})}{\ELL^{4g-3}} + 
\frac{E(\Lambda^2[\ModuliSpace_{1}^{\st}])}{\ELL^{2g}} &=
\frac{E(\ModuliSpace^{\st}_{2})}{\ELL^{4g-3}} + \frac{\ELL^{2g-2}-1}{\ELL^{4g-3}(\ELL^2-1)}E(\Sym^2(\ModuliSpace_{1}^{\st}))
\\&\qquad +
\frac{\ELL^{2g-2}-1}{\ELL^{4g-4}(\ELL^2-1)}E(\Lambda^2[\ModuliSpace_{1}^{\st}])
+ \frac{E(\ModuliSpace_{1}^{\st})}{\ELL^{2g-1}}.
\end{split}
\end{equation}
The RHS of \eqref{eq-IHeqBPS} is equal to the RHS of \eqref{eq-e_stack_minus_sym2} in Theorem~\ref{thm-main_calc}.
Thus 
\begin{equation*}
\begin{split}
\frac{E(\ModuliStack_{2}^\sst)}{\ELL^{4g-4}} &=
\frac{E(\IHc^{\bullet}(\ModuliSpace_{2}^\sst))}{\ELL^{4g-3}}
E(\Hc^\bullet(B\GG_m) \otimes \LL) 
\\&\qquad
+
\frac{1}{\ELL^{2g}}E(\Lambda^2[\ModuliSpace^\st_{1}])
E(\Hc^\bullet(B\GG_m) \otimes \LL)
\\
&\qquad +
\frac{1}{\ELL^{2g}}E(\Sym^2(\Hc^\bullet(\ModuliSpace_{1}^\st)\otimes \H_c^\bullet(B\GG_m) \otimes \LL)).
\end{split}
\end{equation*}
\end{proof}

\section[Application to chi-indpendence checks]{Application to \texorpdfstring{$\chi$}{chi}-independence checks}
\label{sec-comparisons}

Gopakumar--Vafa invariants $n_{g,\beta}$ as defined by Maulik--Toda \cite{maulik2018GopakumarVafaInvariants} are enumerative invariants of one-dimensional sheaves $F$ on a Calabi--Yau 3-fold $X$, which a priori depend on the full Chern character $\operatorname{ch}(F) = (0,0,\beta,\chi)$ of $F$.
The Gopakumar--Vafa invariants $n_{g,\beta}$ are expected to only depend on the curve class $\beta$, i.e., we expect independence of the Euler-characteristic $\chi$, see \cite[Section~3.3]{maulik2018GopakumarVafaInvariants} for more details. 

For a curve $C$, respectively a K3 surface $S$, by considering the local curve $T^\ast C \times \AA^1$, respectively, the local surface $S \times \AA^1$ one defines Gopakumar--Vafa invariants for Higgs bundles on $C$, respectively for the K3 surface $S$, for which $\chi$-independence should hold.

Similarly, BPS-cohomology (when defined) for one-dimensional sheaves on $X$ is expected to be independent of the Euler-characteristic. In fact, $\chi$-independence for BPS-cohomology conjecturally implies $\chi$-independence for Gopakumar--Vafa invariants. 
In this section we show how the formula of Theorem~\ref{thm-general_e-poly} can
be applied to check $\chi$-independence for E-polynomials of  BPS-cohomology.

For work on $\chi$-independence phenomena for BPS-invariants see
\cite{carocci2021BPSInvariantsAdic,
maulik2021CohomologicalChiIndependence,
mellit2020PoincarePolynomialsModuli,kinjo2021CohomologicalChiIndependence}.

\subsection{Higgs bundles}
\label{sec-deg_comparison_higgs}

Let $C$ be a complex smooth connected projective curve of genus 
$g \geq 2$.
We consider the moduli problem of semistable Higgs bundles on $C$.
As in the introduction $\ModuliStack_{r,d}^{\Dol}$ is the moduli stack of rank $r$ degree $d$ Higgs bundles and $\ModuliSpace_{r,d}^{\Dol}$ is the coarse moduli space. Let $\BPS_{r,d}^{\Dol}$ be the BPS cohomology. The BPS cohomology is well-defined by \cite[Theorem~5.16]{kinjo2021CohomologicalChiIndependence}.

We aim to explicitly check $E(\BPS_{2,1}^{\Dol}) = E(\BPS_{2,0}^{\Dol})$, which is already known by \cite[Corollary~5.15]{kinjo2021CohomologicalChiIndependence} (where cohomological $\chi$-independence is shown in general for Gopakumar--Vafa invariants of local curves).

\subsubsection{Rank 2 degree 0 BPS cohomology}
Combining Theorem~\ref{thm-main_e-polynomial} and the computation of the E-polynomial of the stable locus $E(\ModuliSpace_{2,0}^{\Dol,\st})$ in \cite[Theorem~3.7]{kiem2008StringyEfunctionModuli} we determine the E-polynomial $E(\BPS_{2,0}^{\Dol})$ in terms of the genus $g$.
We then simplify the expression to make apparent the equality to $E(\BPS_{2,1}^\Dol)$, which is determined below.

Note that the E-polynomial in \cite[Theorem~3.7]{kiem2008StringyEfunctionModuli}
is for the stable locus of the $\SL_2$-Higgs bundles moduli space. We apply the method explained in 
\cite[Section~4.2]{mauri2021IntersectionCohomologyRank} to convert the  the E-polynomial from the $\SL_2$-Higgs bundles case to the $\GL_2$-Higgs bundles case that we use.

Let $\Jac(C)$ be the Jacobian of the curve $C$. As a first step
write $E(\ModuliSpace_{2,0}^{\Dol,\st})$ so that the contribution of $E(\Jac(C)) = (1-u)^g(1-v)^g$ is clear.

\begin{equation*}
\begin{split}
\frac{E(\ModuliSpace_{2,0}^{\mathrm{st}})}{\ELL^{4g-3}}
&= 
\left(
\frac{(1-u^2v)^g(1-uv^2)^g -\ELL^{g+1}E(\Jac(C))}
{(\ELL-1)^2(\ELL+1)}\right)E(\Jac(C))\\
&\qquad
+\frac{1}{(\ELL-1)(\ELL+1)}
(\ELL E(\Lambda^2[\Jac(C)])+E(\Sym^2(\Jac(C))))
\\&\qquad
+ \frac{(\ELL^g-1)(\ELL^{g-1}-1)}{\ELL^{2g-3}(\ELL-1)(\ELL+1)}
E(\Sym^2(\Jac(C)))
+ \frac{(\ELL^{g-1}-1)(\ELL^{g-2}-1)}{\ELL^{2g-2}(\ELL-1)(\ELL+1)}
E(\Lambda^2[\Jac(C)])
\\&\qquad
+ \frac{(\ELL^{g-1}-1)(\ELL^{g-2}-1)}{\ELL^{g-2}(\ELL-1)}
(E(\Jac(C))^2-E(\Jac(C)))
+ \frac{(\ELL^g-1)(\ELL^{g-1}-1)}{\ELL^{g-1}(\ELL-1)}E(\Jac(C))
\\&\qquad
+\frac{1}{2}E(\Jac(C))\big((1-u)^{g-1}(1-v)^{g-1} 
+ (1+u)^{g-1}(1+v)^{g-1}-2\ELL^{g-1}\big)
\\&\qquad
+E(\Jac(C))\bigg(
\frac{\ELL^{g-1}E(\Jac(C))}{(q-1)^2(q+1)} 
\\&\qquad \qquad
-\frac{(1+u)^{g-1}(1+v)^{g-1}(1-u)(1-v)}{4(q+1)}
- \frac{g-1}{2}\frac{(u+v-2uv)(1-u)^{g-1}(1-v)^{g-1}}{q-1} 
\\&\qquad \qquad
- \frac{4g-7}{4}\frac{E(\Jac(C))}{q-1} 
-\frac{qE(\Jac(C))}{2(q-1)^2}\bigg) 
\end{split}
\end{equation*} 
Now separately gathering terms with factors $E(\Jac(C))^2$,
and a single factor of $E(\Jac(C))=(1-u)^g(1-v)^g$ we have
\begin{equation*}
\begin{split}
\frac{E(\ModuliSpace_{2,0}^{\mathrm{st}})}{\ELL^{4g-3}} &=
\left( \frac{1-q^{g-1}}{q^{g-2}(q-1)}
-\frac{4g-3}{4(q-1)}-\frac{q}{2(q-1)^2}\right)E(\Jac(C))^2
\\&\qquad
+\frac{1}{(\ELL-1)(\ELL+1)}
(\ELL E(\Lambda^2[\Jac(C)])+E(\Sym^2(\Jac(C))))
\\&\qquad
+ \frac{(\ELL^g-1)(\ELL^{g-1}-1)}{\ELL^{2g-3}(\ELL-1)(\ELL+1)}
E(\Sym^2(\Jac(C)))
+ \frac{(\ELL^{g-1}-1)(\ELL^{g-2}-1)}{\ELL^{2g-2}(\ELL-1)(\ELL+1)}
E(\Lambda^2[\Jac(C)])
\\&\qquad
+\bigg( \frac{(1-u^2v)^g(1-uv^2)^g}{(\ELL-1)^2(\ELL+1)}
	-  \frac{(\ELL^{g-1}-1)(\ELL^{g-2}-1)}{\ELL^{g-2}(\ELL-1)}
	+ \frac{(\ELL^g-1)(\ELL^{g-1}-1)}{\ELL^{g-1}(\ELL-1)}
	\\&\qquad\qquad
	+ \frac{1}{2}\big((1-u)^{g-1}(1-v)^{g-1} 
	+ (1+u)^{g-1}(1+v)^{g-1}-2\ELL^{g-1}\big)
	\\&\qquad\qquad
	-\frac{(1+u)^{g-1}(1+v)^{g-1}(1-u)(1-v)}{4(q+1)}
	\\&\qquad\qquad	
	- \frac{g-1}{2}\frac{(u+v-2uv)(1-u)^{g-1}(1-v)^{g-1}}{q-1} 
	 \bigg) E(\Jac(C))
\end{split}
\end{equation*}
Substituting into \eqref{eq-e_stack_minus_sym2} and gathering the $E(\Sym^2(\Jac(C)))$ and $E(\Lambda^{2}[\Jac(C)])$ terms 
we have 
\begin{equation*}
\begin{split}
E(\BPS_{2,0}^\Dol) &= 
\left( \frac{-q^{4g-1}-q^{4g-2}+q^{3g} + q^{3g-1}}{q^{4g-3}(q-1)(q+1)}
-\frac{4g-3}{4(q-1)}-\frac{q}{´2(q-1)^2}\right)E(\Jac(C))^2 
\\&\qquad 
+\frac{q^{4g-1}+q^{4g-2}+q^{4g-3} -q^{3g}-q^{3g-1}}{q^{4g-3}(q-1)(q+1)} E(\Sym^2(\Jac(C)))
\\&\qquad
+\frac{q^{4g-1}+2q^{4g-2}-q^{3g}-q^{3g-1}}{q^{4g-3}(q+1)(q-1)}E(\Lambda^2[\Jac(C)])
\\&\qquad
+\bigg( \frac{(1-u^2v)^g(1-uv^2)^g}{(\ELL-1)^2(\ELL+1)}
\\&\qquad\qquad
+ \frac{1}{2}\big((1-u)^{g-1}(1-v)^{g-1} 
+ (1+u)^{g-1}(1+v)^{g-1}\big)
\\&\qquad\qquad
-\frac{(1+u)^{g-1}(1+v)^{g-1}(1-u)(1-v)}{4(q+1)}
\\&\qquad\qquad	
- \frac{g-1}{2}\frac{(u+v-2uv)(1-u)^{g-1}(1-v)^{g-1}}{q-1} 
\bigg) E(\Jac(C))
\end{split}
\end{equation*}
where two $E(\Jac)$ terms cancelled out with the contribution $E(\ModuliSpace_{1,0}^{\st})/q^{2g-1} = E(\Jac(C))/q^{g-1}$.

Using the identity $E(\Jac(C))^2 = E(\Lambda^2[\Jac(C)]) + E(\Sym^2(\Jac(C)))$ cancels out the first $E(\Jac(C))^2$ term.
\begin{equation*}
	\begin{split}
		E(\BPS_{2,0}^\Dol) &= 
		\left(
		-\frac{4g-3}{4(q-1)}-\frac{q}{2(q-1)^2}\right)E(\Jac(C))^2 
		\\&\qquad 
		+\frac{1}{(q-1)(q+1)} E(\Sym^2(\Jac(C)))
		\\&\qquad
		+\frac{q}{(q+1)(q-1)}E(\Lambda^2[\Jac(C)])
		\\&\qquad
		+\bigg( \frac{(1-u^2v)^g(1-uv^2)^g}{(\ELL-1)^2(\ELL+1)}
		\\&\qquad\qquad
		+ \frac{1}{2}\big((1-u)^{g-1}(1-v)^{g-1} 
		+ (1+u)^{g-1}(1+v)^{g-1}\big)
		\\&\qquad\qquad
		-\frac{(1+u)^{g-1}(1+v)^{g-1}(1-u)(1-v)}{4(q+1)}
		\\&\qquad\qquad	
		- \frac{g-1}{2}\frac{(u+v-2uv)(1-u)^{g-1}(1-v)^{g-1}}{q-1} 
		\bigg) E(\Jac(C))
	\end{split}
\end{equation*}

Expanding 
\begin{equation*}
\begin{split}
E(\Lambda^2[\Jac(C)]) &= \frac{1}{2}(1-u)^{g}(1-v)^{g}\big((1-u)(1-v)(1-u)^{g-1}(1-v)^{g-1} 
\\&\qquad \qquad \qquad \qquad
- (1+u)(1+v)(1+u)^{g-1}(1+v)^{g-1}\big)
\\
E(\Sym^2(\Jac(C))) &= \frac{1}{2}(1-u)^{g}(1-v)^{g}\big((1-u)(1-v)(1-u)^{g-1}(1-v)^{g-1} 
\\&\qquad \qquad \qquad \qquad 
+ (1+u)(1+v)(1+u)^{g-1}(1+v)^{g-1}\big)
\end{split}
\end{equation*}
and combining with the second, third and fourth $E(\Jac(C)) = (1-u)^g(1-v)^g$ terms yields
\begin{equation}
\label{eq-BPS_Higgs_deg0_final}
\begin{split}
E(\BPS_{2,0}^\Dol)&= \frac{(1-u)^g(1-v)^g(1-u^2v)^g(1-uv^2)^g}{(\ELL^2-1)(q-1)}  
+ (1-u)^g(1-v)^g\bigg(
- \frac{(1+u)^g(1+v)^g}{4(q+1)}
 \\
& \qquad
-\frac{g}{2} \frac{(u+v-2uv)(1-u)^{g-1}(1-v)^{g-1}}{q-1} 
- \frac{4g-3}{4}\frac{(1-u)^g(1-v)^g}{q-1} 
 \\
& \qquad
- \frac{1}{2}\frac{q(1-u)^g(1-v)^g}{(q-1)^2}\bigg).
\end{split}
\end{equation}
\subsubsection{Rank 2 degree 1 BPS cohomology}
In degree 1 the BPS cohomology is the shifted cohomology of the 
rank 2 degree 1 coarse moduli space.
A formula for its E-polynomial can be extracted from
\cite[Appendix]{garcia-prada2014MotivesModuliChains} and is given by
\begin{equation*}
\begin{split}
E(\BPS_{2,1}^\Dol) &=\frac{E(\ModuliSpace_{2,1})}{\ELL^{4g-3}}\\
&= E(\Jac(C))\left(\frac{(1-u^2v)^g(1-uv^2)^g-\ELL^g E(\Jac)}
{(\ELL^2-1)(\ELL-1)}+ \sum_{d=1}^{g-1}E(\Sym^{2g-2d-1}(C))\right)\\
&=
\frac{(1-u)^g(1-v)^g(1-u^2v)^g(1-uv^2)^g}{(\ELL^2-1)(q-1)} 
- \frac{\ELL^g(1-u)^{2g}(1-v)^{2g}}{(\ELL^2-1)(q-1)} \\
& \qquad
+ (1-u)^g(1-v)^g\sum_{d=1}^{g-1}
\Coeff_{t^{2g-2d-1}}\left(\frac{(1-ut)^g(1-vt)^g}{(1-t)(1-qt)}\right).
\end{split}
\end{equation*}
The third equality follows from Macdonald's computation of the cohomology of symmetric powers of curves \cite{macdonald1962SymmetricProductsAlgebraic}.
We evaluate 
\begin{equation*}
\begin{split}
\sum_{d=1}^{g-1}
&\Coeff_{t^{2g-2d-1}}\left(\frac{(1-ut)^g(1-vt)^g}{(1-t)(1-qt)}\right)
= \frac{\ELL^g(1-u)^g(1-v)^g}{(q-1)^2(q+1)} - \frac{(1+u)^g(1+v)^g}{4(q+1)}
\\&
-\frac{g}{2} \frac{(u+v-2uv)(1-u)^{g-1}(1-v)^{g-1}}{q-1} 
- \frac{4g-3}{4}\frac{(1-u)^g(1-v)^g}{q-1} 
- \frac{1}{2}\frac{q(1-u)^g(1-v)^g}{(q-1)^2}
\end{split}
\end{equation*}
following Hitchin 
\cite[Proof of Theorem~7.6]{hitchin1987SelfDualityEquationsRiemann}
(see also \cite[Section~3.3]{kiem2008StringyEfunctionModuli}).
Altogether we have
\begin{equation*}
\begin{split}
E(\BPS_{2,1}^\Dol)
&=
\frac{(1-u)^g(1-v)^g(1-u^2v)^g(1-uv^2)^g}{(\ELL^2-1)(q-1)} - \frac{\ELL^g(1-u)^{2g}(1-v)^{2g}}{(\ELL^2-1)(q-1)} 
\\& \qquad
+ (1-u)^g(1-v)^g\bigg(
\frac{\ELL^g(1-u)^g(1-v)^g}{(q-1)^2(q+1)} - \frac{(1+u)^g(1+v)^g}{4(q+1)}
\\&\qquad 
-\frac{g}{2} \frac{(u+v-2uv)(1-u)^{g-1}(1-v)^{g-1}}{q-1} 
- \frac{4g-3}{4}\frac{(1-u)^g(1-v)^g}{q-1} 
\\&\qquad
- \frac{1}{2}\frac{q(1-u)^g(1-v)^g}{(q-1)^2}\bigg).
\end{split}
\end{equation*}
This simplifies slightly to
\begin{equation*}
\begin{split}
E(\BPS_{2,1}^\Dol)&= \frac{(1-u)^g(1-v)^g(1-u^2v)^g(1-uv^2)^g}{(\ELL^2-1)(q-1)}  
+ (1-u)^g(1-v)^g\bigg(
- \frac{(1+u)^g(1+v)^g}{4(q+1)}
 \\
& \qquad
-\frac{g}{2} \frac{(u+v-2uv)(1-u)^{g-1}(1-v)^{g-1}}{q-1} 
- \frac{4g-3}{4}\frac{(1-u)^g(1-v)^g}{q-1} 
 \\
& \qquad
- \frac{1}{2}\frac{q(1-u)^g(1-v)^g}{(q-1)^2}\bigg),
\end{split}
\end{equation*}
which agrees with the expression \eqref{eq-BPS_Higgs_deg0_final} 
for $E(\BPS_{2,0}^\Dol)$. 

\subsubsection{Betti side}%
Via non-abelian Hodge theory for stacks, as developed in \cite{davison2021NonabelianHodgeTheory}, one similarly expects a $\chi$-independence phenomenon on the Betti side. 

Since $\gcd(2,1) = 1$, we have $E(\BPS_{2,1}^{\Betti}) =q^{3-4g} E(\ModuliSpace_{2,1}^{\Betti})$.
The E-polynomial $E(\ModuliSpace_{2,1}^{\Betti})$ was determined in \cite[Corollary~3.6.1]{hausel2008MixedHodgePolynomials} and 
the intersection E-polynomial $IE(\ModuliSpace_{2,0}^{\Betti})$ was determined 
in \cite[Theorem~1.4]{mauri2021IntersectionCohomologyRank}.
Using Theorem~\ref{thm-main_e-polynomial} one can directly check $E(\BPS_{2,0}^{\Betti}) = E(\BPS_{2,1}^{\Betti})$.

\subsection{Sheaves on K3 surfaces}

We can reinterpret some results in \cite{decataldo2021HodgeNumbersGrady} 
as a cohomological $\chi$-independence check for sheaves on K3 surfaces.

Let $S$ be a K3 surface. 
Let $H$ be a sufficiently general polarization of $S$.
Suppose $(S,H)$ is of genus 2, i.e.,
the curves in the linear system $\lvert H \rvert$ are of genus 2.
Let $v\in H^\bullet_{\mathrm{alg}}(S,\ZZ)$ be a primitive Mukai vector with $v^2 = 2$. 
Consider the moduli stack and moduli spaces 
of $H$-Gieseker-semistable sheaves $\ModuliStack_{S,2v}^{\Hsst}$,
$\ModuliSpace_{S,2v}^{\Hsst}$,
and  $\ModuliSpace^{\Hsst}_{S,v}$. %
Let $OG10$ be O'Grady's ten-dimensional sporadic example of a hyper-K\"ahler manifold. By \cite[Lemma~4.1.3]{decataldo2021HodgeNumbersGrady} we have 
\begin{equation*}
IE(\ModuliSpace_{S,2v}^{\Hsst}) =
E(OG10) - \ELL E(\Sym^2(\ModuliSpace_{S,v}^{\Hsst}))
- \ELL^{3}H(\ModuliSpace_{S,v}^{\Hsst}).
\end{equation*}
and by \eqref{eq-e_stack_minus_sym2} we have 
\begin{equation}
\label{eq-BPSvsOG10}
E(\BPS_{S,2v}) =
\frac{E(OG10)}{\ELL^{5}} -\frac{E_{u^2,v^2}(\ModuliSpace_{S,v}^{\Hsst})}{\ELL^4} - 
\frac{E(\ModuliSpace_{S,v}^{\Hsst})}{\ELL^2}.
\end{equation}

Pick a primitive Mukai vector 
$w \in\H^{\bullet}_{\mathrm{alg}}(S,\ZZ)$
satisfying $w^2 = 4 = (2v)^2$.
The moduli space $\ModuliSpace_{S,w}^{\Hsst}$ is smooth 
and deformation equivalent to the Hilbert scheme of 5 points on a K3 surface. 
\begin{corollary}\label{cor-k3chi}
Suppose the cohomological integrality conjecture is true for 
$\CA_{\Coh(S),v}^{\Hsst}$, then
 $$E(\BPS_{S,2v}) =E(\BPS_{S,w}).$$
\end{corollary}
\begin{proof}
By definition we have $E(\BPS_{S,w}) = q^{-5}E(\ModuliSpace_{S,w}^{H-\sst})$. By \cite[Proposition~6.1.2]{decataldo2021HodgeNumbersGrady} and \eqref{eq-BPSvsOG10} we have
$E(\BPS_{S,2v})=q^{-5}E(\ModuliSpace_{S,w}^{H-\sst})$.
\end{proof}
Therefore, the polynomial $E(\BPS_{S,2v})$ is determined by 
the Hodge numbers of $\ModuliSpace_{S,w}^{\Hsst}$, 
which are recorded in \cite[(103)]{decataldo2021HodgeNumbersGrady}.

\begin{remark}
If we assume the $\chi$-independence conjecture for BPS cohomology, then equation~\eqref{eq-BPSvsOG10} (which follows from a simple application of the decomposition theorem \cite[Lemma~4.1.3]{decataldo2021HodgeNumbersGrady} and Theorem~\ref{thm-main_calc}) together with $\chi$-independence, yields a conjectural computation of the Hodge numbers of $OG10$. 
\end{remark}

\printbibliography
\end{document}

%% file: preamble.tex
\usepackage[
style=alphabetic,
doi=false,
isbn=false,
url=false,
eprint=false,
block = space,
sorting=anyt]{biblatex} 
\newtheorem{theorem}{Theorem}[section]
\newtheorem{corollary}[theorem]{Corollary}
\newtheorem{proposition}[theorem]{Proposition}
\newtheorem{lemma}[theorem]{Lemma}
\newtheorem{conjecture}[theorem]{Conjecture}

\theoremstyle{definition} 
\newtheorem{definition}[theorem]{Definition}
\newtheorem{theorem-definition}[theorem]{Theorem-Definition}

\newtheorem*{notation*}{Notation}
\newtheorem{setting}{Setting}

\theoremstyle{remark} 
\newtheorem{remark}[theorem]{Remark}
\newtheorem{example}[theorem]{Example}

\newcommand{\longto}{\longrightarrow} 
\newcommand{\isoto}{\stackrel{\sim}{\to}}

\newcommand{\into}{\hookrightarrow}
\newcommand{\longinto}{\lhook\joinrel\longrightarrow}
\newcommand{\onto}{\rightarrow\mathrel{\mkern-14mu}\rightarrow}

\let\smallbar\bar
\renewcommand{\bar}{\overline}
\renewcommand{\subset}{\subseteq}

\renewcommand{\setminus}{\smallsetminus}

\newcommand*\circled[1]{\tikz[baseline=(char.base)]{
		\node[shape=circle,draw,inner sep=2pt] (char) {#1};}}

\renewcommand{\AA}{\mathbb{A}}

\newcommand{\BPS}{\mathrm{BPS}}
\newcommand{\Betti}{\mathrm{Betti}}
\newcommand{\CA}{\mathcal{A}}
\newcommand{\CC}{\mathbb{C}}

\newcommand{\CF}{\mathcal{F}}
\newcommand{\CE}{\mathcal{E}}
\newcommand{\CG}{\mathcal{G}}

\newcommand{\CK}{\mathcal{K}}

\newcommand{\CO}{\mathcal{O}}
\newcommand{\CQ}{\mathcal{Q}}
\newcommand{\CV}{\mathcal{V}}
\newcommand{\Coeff}{\operatorname{Coeff}}
\newcommand{\CoHa}{\mathcal{A}_{0}}
\newcommand{\Coh}{\operatorname{Coh}}
\newcommand{\coarsemap}{p}

\newcommand{\DiagSpace}{\Omega}
\newcommand{\DiagStack}{\mathfrak{Z}}
\newcommand{\DiagNonSplit}{\smallbar{\DiagStack}}
\newcommand{\DiagSplit}{\ddot{\DiagStack}}
\newcommand{\Dol}{\mathrm{Dol}}
\newcommand{\EGLtwo}{q(q-1)^2(q+1)}
\newcommand{\EGm}{q-1}
\newcommand{\ELL}{q}
\newcommand{\EPP}{q+1}
\newcommand{\End}{\operatorname{End}}
\newcommand{\Ext}{\operatorname{Ext}}

\newcommand{\ExactNonSplit}{\smallbar{\ExactStack}}
\newcommand{\ExactSplit}{\ddot{\ExactStack}}
\newcommand{\ExactStack}{\mathfrak{X}}
\newcommand{\FB}{\mathfrak{B}}
\newcommand{\FI}{\mathfrak{I}}

\newcommand{\FX}{\mathfrak{X}}

\newcommand{\FreeLie}{\operatorname{FreeLie}}
\newcommand{\GG}{\mathbb{G}}
\newcommand{\GL}{\mathrm{GL}}
\newcommand{\Gr}{\operatorname{Gr}}
\renewcommand{\H}{H}
\newcommand{\Hc}{H_c}

\newcommand{\Hom}{\operatorname{Hom}}
\newcommand{\Hsst}{H\text{-sst}}
\newcommand{\Hst}{H\text{-st}}

\newcommand{\IHc}{IH_c}
\newcommand{\IsoPairs}{\mathfrak{D}}

\newcommand{\Jac}{\operatorname{Jac}}
\newcommand{\LL}{\mathbb{L}}
\newcommand{\QQ}{\mathbb{Q}}

\newcommand{\MHS}{\mathbf{MHS}}

\newcommand{\MHSboundedabove}{\MHS^{-}}
\newcommand{\ModuliSpace}{\mathcal{M}}
\newcommand{\ModuliStack}{\mathfrak{M}}
\newcommand{\MSym}{\Sym}
\newcommand{\MAlt}{\Lambda}

\newcommand{\NonIsoPairs}{\mathfrak{U}}
\newcommand{\Ob}{\operatorname{Ob}}
\newcommand{\OffDiagStack}{\mathfrak{Y}}
\newcommand{\OffDiagNonSplit}{\smallbar{\OffDiagStack}}
\newcommand{\OffDiagSplit}{\ddot{\OffDiagStack}}
\newcommand{\PP}{\mathbb{P}}
\newcommand{\pr}{\mathrm{pr}}

\newcommand{\R}{\operatorname{\mathbf{R}} \kern -2pt}

\newcommand{\RSHom}{\mathbf{R} \kern -3pt \SHom }

\newcommand{\relSpec}{\operatorname{\mathcal{S} \kern -1pt \mathit{pec}}}
\newcommand{\SHom}
	{\operatorname{\mathcal{H} \kern -2pt \mathit{o \kern -1pt m}}}
\newcommand{\SingNonSplit}{\smallbar{\SingStack}}
\newcommand{\SingSpace}{\Sigma}
\newcommand{\SingSplit}{\ddot{\SingStack}}
\newcommand{\SingStack}{\mathfrak{S}}
\newcommand{\SL}{\operatorname{SL}}

\newcommand{\Stacks}{\mathbf{St}}

\newcommand{\Sym}{\operatorname{Sym}}

\newcommand{\sC}{\mathscr{C}}

\newcommand{\sst}{\mathrm{ss}}
\newcommand{\st}{\mathrm{s}}
\newcommand{\Tot}{\operatorname{Tot}}
\newcommand{\tRHom}
	{\tau \mathcal{E} \kern -2pt \mathit{x\kern-1pt t}}

\newcommand{\Varieties}{\mathbf{Var}}
\newcommand{\vdim}{\operatorname{vdim}}
\newcommand{\ZZ}{\mathbb{Z}}